\documentclass[12pt]{amsart}
\usepackage[all]{xy}

\usepackage{amsthm,amsfonts, amssymb, amscd}
\usepackage{mathtools}
\usepackage{youngtab}
\usepackage{ytableau}
\usepackage[all]{xy}
\usepackage{subcaption}
\usepackage{caption}
\usepackage{euscript}
\usepackage{cite}
\usepackage{booktabs, tabularx}
\newcolumntype{C}{>{\centering\arraybackslash}X}

\usepackage{tikz}
\usetikzlibrary{matrix, arrows, decorations.markings}
\usetikzlibrary{decorations.pathmorphing}
\usetikzlibrary{positioning}

\usepackage{pgfplots}
\pgfplotsset{compat=1.18}

\usepackage{adjustbox}

\usepackage[normalem]{ulem}

\usepackage{hyperref}
\hypersetup{%
    linktoc=page
}

\let\oldtocsection=\tocsection
\let\oldtocsubsection=\tocsubsection
\renewcommand{\tocsection}[2]{\hspace{0em}\oldtocsection{#1}{#2}}
\renewcommand{\tocsubsection}[2]{\hspace{1em}\oldtocsubsection{#1}{#2}}
\DeclareEmphSequence{\color{blue}\itshape}

\newtheorem{thm}{Theorem}[section]
\newtheorem{lemma}[thm]{Lemma}
\newtheorem{cor}[thm]{Corollary}
\newtheorem{prop}[thm]{Proposition}

\newtheorem{conjecture}[thm]{Conjecture}
\theoremstyle{remark}
\newtheorem{rem}[thm]{Remark}
\theoremstyle{remark}
\newtheorem{example}[thm]{Example}
\theoremstyle{definition}

\theoremstyle{definition}

\theoremstyle{definition}

\newtheorem{definition}[thm]{Definition}

\numberwithin{equation}{section}

\newcommand{\C}{\mathbb{C}}           % Use for complex numbers.
\newcommand{\Z}{\mathbb{ Z}}           % Use for integers.
 
\newcommand{\Ad}{\operatorname{Ad }}             % Use for Adjoint action
\newcommand{\ch}{\operatorname{ch}}

\newcommand{\Inv}{\operatorname{Inv}}

\newcommand{\fb}{{\mathfrak b}}

\newcommand{\fg}{{\mathfrak g}}

\newcommand{\fl}{{\mathfrak l}}

\newcommand{\ft}{{\mathfrak t}}

\newcommand{\fz}{\mathfrak z}

\newcommand{\gG}{\Gamma}

 \newcommand{\cb}{\mathcal{B}}

 \newcommand{\cp}{\mathcal{P}}

\newcommand{\vx}{\mathbf{x}}

\renewcommand{\tilde}{\widetilde}

\renewcommand{\bar}[1]{\overline{#1}}

\usepackage[colorinlistoftodos]{todonotes}

\definecolor{melon}{rgb}{0.99,0.72,0.67}
\definecolor{cornflower}{rgb}{0.66,0.77,0.95}
\definecolor{lightgreen}{rgb}{0.67, 0.99, 0.67}
\definecolor{aqua}{rgb}{0, 1, 1}

\newcommand{\Hess}{{\mathcal B}}
\newcommand{\Poin}{\mathsf{P}}

\newcommand{\asc}{\mathsf{a}\mathsf{s}\mathsf{c}}
\newcommand{\fch}{\mathsf{c}\mathsf{h}}
\newcommand{\sgn}{\mathsf{s}\mathsf{g}\mathsf{n}}
\newcommand{\pn}{\mathsf{P}\mathsf{a}\mathsf{r}}

\newcommand{\xdownarrow}[1]{%
  {\left\downarrow\vbox to #1{}\right.\kern-\nulldelimiterspace}
}

\begin{document}

%%%%%%%%%%%%%%%%%%%%%%%%%%%%%%%%%%%%%%%%%%%%%%%%%%%%
%%%%%%%%%%%%%%%%%%%%%%%%%%%%%%%%%%%%%%%%%%%%%%%%%%%%
%%%%%%%%%%%%%%%%%%%%%%%%%%%%%%%%\end{center}%%%%%%%%%%%%%%%%%%%%
\title[Restrictions of the dot action representation]{Restrictions of the dot action representation}

\author{Hsin-Chieh Liao}
\address{Department of Mathematics and Statistics\\ Washington University in St. Louis \\ One Brookings Drive\\ St. Louis, Missouri\\ 63130\\ USA  }
\email{liaoh@wustl.edu}

\author{Martha Precup}
\address{Department of Mathematics and Statistics\\ Washington University in St. Louis \\ One Brookings Drive\\ St. Louis, Missouri\\ 63130\\ USA  }
\email{martha.precup@wustl.edu}

\author{John Shareshian}
\address{Department of Mathematics and Statistics\\ Washington University in St. Louis \\ One Brookings Drive\\ St. Louis, Missouri\\ 63130\\ USA  }
\email{jshareshian@wustl.edu}

\date{\today}

\begin{abstract} Regular semisimple Hessenberg varieties are subvarieties of flag varieties defined for any reductive group. Their cohomology carries a representation of the associated Weyl group, known as the dot action representation. In type $A$, these representations can be computed using the chromatic quasisymmetric functions of naturally labeled unit interval graphs.

We prove that restricting the dot action representation of any Weyl group to a type $A$ parabolic subgroup yields a representation isomorphic to a direct sum of type $A$ dot action representations. This work is motivated by the conjecture that an analogous property holds for all parabolic subgroups. In particular, for classical Weyl groups of types $B/C$ and $D$, we use chromatic symmetric functions to provide a streamlined, combinatorial formula for the restriction to a specific maximal type $A$ parabolic subgroup.

As an application, we prove a conjecture of Lesnevich regarding the representations associated with ideals of types $B$ and $C$ identified under the standard poset isomorphism. Specifically, we show that the corresponding dot action representations are isomorphic if and only if the ideals determine the same reflection sets in their respective root systems. Furthermore, we show that a Catalan number counts the total number of such ideals.
\end{abstract}

\maketitle

%%%%%%%%%%%%%%%%%%%%%%%%

\section{Introduction: Some history, notation, and statement of results}

We aim to understand better Tymoczko's dot action representations of Weyl groups on the cohomology of regular semisimple Hessenberg varieties in all Lie types.  Combining work of Athanasiadis in \cite{Athanasiadis2015}, Brosnan--Chow in \cite{Brosnan-Chow2018}, and Shareshian--Wachs in \cite{Shareshian-Wachs2016} gives a satisfactory result in type $A$ for the symmetric group, but the general case remains open.  Here we take advantage of the fact that if $\Phi$ is an irreducible root system other than type $F_4$, the associated Weyl group $W$ has a maximal parabolic subgroup $W_J$ of type $A$.  We determine the restriction of the dot action representation to $W_J$ in all such cases.
Our main theorem is that such a restriction is a direct sum (with shifted gradings) of dot action representations of $W_J$ on cohomology rings of various type $A$ Hessenberg varieties, hence we can again apply results from \cite{Athanasiadis2015,Brosnan-Chow2018,Shareshian-Wachs2016}.

\subsection{A result in the classical cases} \label{sec.classical} First we present a specific version of our main result that is somewhat easier to introduce and state than the general version.  We assume temporarily that the reader is familiar with various definitions and results about dot action representations of Weyl groups and defer definitions to later sections.  We take $\Phi$ to be one of the classical root systems $B_n,C_n,D_n$, with $n \geq 2$ for types $B_n$ and $C_n$ and $n \geq 4$ for type $D_n$.  We assume that $\Phi$ contains a fixed set $\Delta=\{\alpha_i:i \in [n]\}$ of simple roots such that
$$
\alpha_i=\left\{ \begin{array}{ll} e_i-e_{i+1} & i \in [n-1], \\ e_n & i=n \mbox{ and } \Phi=B_n, \\ 2e_n & i=n \mbox{ and } \Phi=C_n, \\ e_{n-1}+e_n & i=n \mbox{ and } \Phi=D_n, \end{array} \right.
$$
and define the respective sets $\Phi^+$ and $\Phi^-$ of positive and negative roots accordingly. The Weyl group $W=W(\Phi)$ is generated by the set $\{s_\beta:\beta \in \Delta\}$ of simple reflections.  If 
\[
J=\{\alpha_{1}, \ldots, \alpha_{n-1}\}  \subset \Delta
\]
is the subset of simple roots obtained by deleting $\alpha_n$ from $\Delta$, then $W_J:=\langle s_\beta:\beta \in J \rangle$ is a maximal parabolic subgroup of $W$ isomorphic to $S_n$.

Let $\fg$ be the simple Lie algebra associated to $\Phi$. Associated to each lower order ideal $I \subseteq \Phi^+$ is a Hessenberg space $H=H(I) \subseteq \fg$.  Such $H$ and any regular semisimple $s \in \fg$ determine the regular semisimple Hessenberg variety $\Hess(s,H)$.  In \cite{Tymoczko2008}, Tymoczko defined the dot action representation of $W$ on $H^\ast(\Hess(s,H))$.  We are interested in determining the restriction to $W_J$ of the dot action representation of $W$, more precisely, the graded representation
$$
\left(\sum_{j \geq 0} H^{2j}(\Hess(s,H(I)))t^j\right)\Big\downarrow^W_{W_J}.
$$

Fix $H=H(I)$ and let ${^J}W$ be the set of shortest right coset representatives for $W_J$ in $W$.  Given $v \in {^J}W$, we make two definitions:
\begin{itemize}
\item the \emph{$H$-inversion number $\ell_H(v)$} is the number of $\alpha \in I$ such that $v(\alpha) \in \Phi^-$, and
\item the graph $\Gamma_{H,v}$ has vertex set $[n]$ and edge set $$E_{H,v}:=\{ij:1 \leq i<j \leq n \mbox{ and } v^{-1}(e_i-e_j) \in I\}.$$
\end{itemize}

We write $X_\Gamma(\vx;t)$ for the chromatic quasisymmetric function of a graph $\Gamma=([n],E)$, as defined in \cite{Shareshian-Wachs2016}, and recall that if $\Gamma$ is a naturally labeled unit interval graph then $X_\Gamma(\vx;t)$ is a symmetric function (see \cite[Definition 1.2 and Theorem 4.5]{Shareshian-Wachs2016}).  Also, we write $\omega$ for the involution on the algebra of symmetric functions that exchanges each complete homogeneous function $h_k$ with the corresponding elementary symmetric function $e_k$.

\begin{thm} \label{bcd}
Assume the root system $\Phi$ is of type $B_n$, $C_n$, or $D_n$ with associated Weyl group $W$ and that $W_J$ is the maximal parabolic subgroup of $W$ isomorphic to $S_n$ defined above.  
Let $I\subseteq \Phi^+$ be an order ideal with associated Hessenberg space $H=H(I)$. 
Then
\begin{enumerate}
\item for all $v\in {^J}W$, the graph $\Gamma_{H,v}$ is a naturally labeled unit interval graph and
\item the Frobenius characteristic of the restriction of the dot action representation of $W$ on $H^\ast(\Hess(s,H))$ to $W_J$ satisfies
$$
\fch\left(\left(\sum_{j \geq 0} H^{2j}(\Hess(s,H))t^j\right)\Big\downarrow^W_{W_J}\right)=\omega\left(\sum_{v \in {^J}W}t^{\ell_H(v)}X_{\Gamma_{H,v}}(\vx;t)\right).
$$
\end{enumerate}
\end{thm}

Combinatorial formulas for the Schur and power sum decompositions of $X_\Gamma(\vx;t)$ when $\Gamma$ is a naturally labeled unit interval graph were determined, respectively, by Shareshian and Wachs in \cite[Theorem 6.3]{Shareshian-Wachs2016} and by Athanasiadis in \cite[Theorem 4]{Athanasiadis2015} (see also \cite[Theorem 7.7]{Shareshian-Wachs2016}). We will provide in Section \ref{bcdproof} a graph-theoretic method for finding both $\Gamma_{H,v}$ and $\ell_H(v)$ for all $v \in {^J}W$.  Thus we obtain, through the inverse Frobenius characteristic, combinatorial formulas for the irreducible decompositions and character values of the restrictions of dot action representations from $W$ to $W_J$ in all classical types.

Recall that the root posets of type $B_n$ and $C_n$ are isomorphic (though the root systems are not). In particular, there is a natural bijection between lower order ideals in $\Phi^+(B_n)$ and $\Phi^+(C_n)$.  One can ask when the corresponding dot action representations of type $B_n$ and $C_n$ are isomorphic.
As an application of our graph-theoretic method of computing the graded Frobenius characteristic in Theorem~\ref{bcd}, we confirm a conjecture by Lesnevich \cite[Remark 3.2]{Lesnevich2025} that this is the case if and only if the type $B$ and $C$ ideals determine the same reflection sets (see Theorem \ref{thm.nate} below).

\subsection{History, definitions and notation} \label{genint} Let us now turn to the general case.  We assume some familiarity with algebraic groups and symmetric functions.  Our definitions and notation are for the most part standard.  A reader familiar with recent work on the dot action representation can skip this section, although the last two paragraphs might provide insight into our thoughts.  

Let $G$ be a reductive complex algebraic group with Borel subgroup $B$ and maximal torus $T \leq B$. Let $W=N_G(T)/T$ be the associated Weyl group.  Write $\fg$, $\fb$, and $\ft$ for the respective Lie algebras of $G$, $B$, and $T$. We denote the adjoint representation of $G$ on $\fg$ by $\Ad(g)(x)$ for $g\in G$ and $x\in \fg$.  

Our choice of $T$ determines a Cartan decomposition
$$
\fg=\ft \oplus \bigoplus_{\alpha \in \Phi}\fg_\alpha,
$$
where $\Phi$ is the root system determined up to isomorphism by $G$ and $\fg_\alpha$ is the root space in $\fg$ associated to $\alpha\in \Phi$.  We choose a set $\Delta$ of simple roots in $\Phi$ and write $\Phi^+$ and $\Phi^-$, respectively, for the sets of positive and negative roots in $\Phi$ with respect to this choice.  We assume that
$$
\fb=\ft \oplus \bigoplus_{\alpha \in \Phi^+} \fg_\alpha.
$$
The set $\Phi^+$ is partially ordered, with $\beta$ covering $\alpha$ if $\beta-\alpha \in \Delta$. We say a subset $I\subseteq \Phi^+$ is a \emph{(lower) order ideal} if $\beta\in I$ implies that all $\alpha$ below $\beta$ in this partial order are contained in $I$. An example of such an $I$ in the type $A_3$ root system appears in Figure~\ref{fig:A3}. The number of such ideals in $\Phi^+$ for $\Phi$ an irreducible root system is precisely the $W$-Catalan number~\cite{Cellini-Papi}. 

A \emph{Hessenberg space} (with respect to $\fb$) is any vector subspace $H \subseteq \fg$ 
that contains $\fb$ and is $\Ad(B)$-invariant. Equivalently, $H$ is a Hessenberg space if there is some order ideal $I$ in $\Phi^+$ such that
$$
 H=H(I):=\fb \oplus \bigoplus_{\alpha \in I}\fg_{-\alpha}.
$$
We sometimes write $\Phi_H\coloneqq\{\gamma\in\Phi : \fg_{\gamma}\in H\}=\Phi^+\cup(-I)$ and $\Phi_H^-\coloneqq \Phi_H\cap \Phi^-=-I$.
 
Given $s \in \fg$ and Hessenberg space $H$, the \emph{Hessenberg variety} $\Hess(s,H)$ is a subvariety of the flag variety $G/B$,
$$
\Hess(s,H):=\{gB \in G/B:\Ad(g^{-1})s \in H\}.
$$
From now on, unless otherwise stated, we will assume that $s$ is a regular semisimple element of $\fg$.  Under this assumption, $\Hess(s,H)$ is called a \emph{ regular semisimple Hessenberg variety}.  Such varieties were studied first in the case $G=SL_n(\C)$ by de Mari and Shayman in \cite{DS1988}.  The general case was studied first by de Mari, Procesi, and Shayman in \cite{DPS1992}.  

We may assume that $s \in \ft$, i.e., that the centralizer $C_G(s)$ of $s$ under the adjoint action is $T$.  It follows quickly that $T$ acts on $\Hess(s,H)$ by translation, that is, if $gB \in \Hess(s,H)$ then $tgB \in \Hess(s,H)$ for all $t \in T$. This action of $T$ on $\Hess(s,H)$ is equivariantly formal, as defined by Goresky, Kottwitz, and MacPherson in \cite{GKM1998}.  Thus, the (complex) $T$-equivariant and ordinary cohomology rings of $\Hess(s,H)$ can be represented using functions from the vertex set of the associated moment graph (or GKM graph) to the symmetric algebra $\C[\ft^\ast]$ on the dual space of $\ft$.  While $\Hess(s,H)$ typically admits no obvious action of the Weyl group $W:=N_G(T)/T$, it turns out that the rings of functions in question admit such actions, defined through an action of $W$ on the GKM graph.  Thus we obtain a representation of $W$ on each cohomology group of $\Hess(s,H)$.  Collecting these representations, one obtains a graded representation of $W$ known as the \emph{dot action representation}, introduced and first studied by Tymoczko in~\cite{Tymoczko2008}.

If one accepts that a representation is understood if its character or irreducible decomposition is known, then we have a good understanding of the dot action representation in type $A$, that is, when $[G/Z(G),G/Z(G)]=PSL_n(\C)$. This description relies on the combinatorics of chromatic symmetric functions. We now establish the terminology needed to state the relevant results.

In the type $A_{n-1}$ case, $W=S_n$ and $\Phi^+=\{e_i-e_j:1\leq i<j \leq n\}$.  Also, $I \subseteq \Phi^+$ is an ideal if and only if whenever $e_i-e_j \in I$ and $i<k<j$, both $e_i-e_k$ and $e_k-e_j$ lie in $I$.  The graph $\gG(I)$ with vertex set $[n]$ and edge set $E(I):=\{ij:e_i-e_j \in I\}$ is called a \emph{naturally labeled unit interval graph} and determines $I$ uniquely; see Figure~\ref{fig:A3}.

\begin{figure}[t]
    \centering
    \begin{subfigure}{0.35\textwidth}
        \centering
        \begin{tikzpicture}[
        vertex/.style={font=\scriptsize}
        ]
        \def\df{20}
        \def\w{1 cm}

        % Coordinates for the Type A3 root poset (triangular shape)
        \coordinate (12) at (0,0);
        \coordinate (23) at (2,0);
        \coordinate (34) at (4,0);
        \coordinate (13) at (1,1);
        \coordinate (24) at (3,1);
        \coordinate (14) at (2,2);
        
        % Highlighted Order Ideal I = {e1-e2, e2-e3, e1-e3}
        \path[fill=blue!\df, draw=blue!\df, line width=\w, line cap=round, line join=round] 
        (12) to (13) to (23) to (34) to (12);

        % Root labels
        \node[vertex] (p12) at (12) {$e_1-e_2$};
        \node[vertex] (p23) at (23) {$e_2-e_3$};
        \node[vertex] (p34) at (34) {$e_3-e_4$};
        \node[vertex] (p13) at (13) {$e_1-e_3$};
        \node[vertex] (p24) at (24) {$e_2-e_4$};
        \node[vertex] (p14) at (14) {$e_1-e_4$};
    
        % Poset relations (Hasse diagram edges)
        \draw[very thick] (p12)--(p13)--(p14);
        \draw[very thick] (p23)--(p13);
        \draw[very thick] (p23)--(p24)--(p14);
        \draw[very thick] (p34)--(p24);

        \end{tikzpicture}
    \end{subfigure}
    \hspace{1.5 cm}
    \begin{subfigure}{0.35\textwidth}
        \centering
        \begin{tikzpicture}[
                scale=0.8,
                vertex/.style={circle, draw, minimum size=0.7 cm},
                every edge/.style={draw, thick}
            ]
            % Vertices for Type A3 (indices 1, 2, 3, 4)
            \node[vertex] (v1) at (0,0) {\small $1$};
            \node[vertex] (v2) at (1.5,0) {\small $2$};
            \node[vertex] (v3) at (3,0) {\small $3$};
            \node[vertex] (v4) at (4.5,0) {\small $4$};
            
            % Edges corresponding to roots in the ideal I: 
            \draw (v1)--(v2);
            \draw (v2)--(v3);
            \draw (v3)--(v4);
            \draw (v1) to [out=45, in=135] (v3);
        \end{tikzpicture}   
    \end{subfigure}
\caption{An order ideal $I$ in the type $A_3$ root system and graph $\Gamma(I)$}
\label{fig:A3} 
\end{figure}
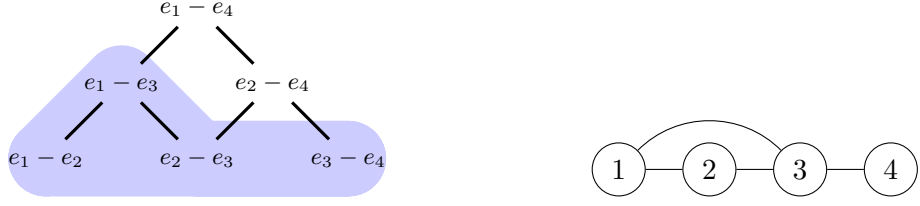

The \emph{chromatic quasisymmetric function} of a graph $\gG=([n],E)$ is a generating function for proper colorings of $\gG$, weighted by ascents.  Here, a proper coloring is a function $\kappa:[n] \rightarrow \Z_{>0}$ such that $\kappa(i) \neq \kappa(j)$ if $ij \in E$, and an \emph{ascent} of $\kappa$ is any $ij \in E$ such that $i<j$ and $\kappa(i)<\kappa(j)$.  We write $\asc(\kappa)$ for the number of ascents of $\kappa$.  The chromatic quasisymmetric function of $\gG$ is
$$
X_{\gG}(\vx;t):=\sum_\kappa t^{\asc(\kappa)} \prod_{v \in [n]}x_{\kappa(v)},
$$
where the sum is taken over all proper colorings of $\gG$. The expression $X_{\gG}(\vx;t)$ is
a quasisymmetric function in the set $\vx=\{x_i:i \in \Z_{>0}\}$ of pairwise commuting variables.   The chromatic quasisymmetric function refines (through $t$) the chromatic symmetric function defined by Stanley in \cite{Stanley1995} and was introduced by the third author and Wachs in \cite{Shareshian-Wachs2016}.  

Suppose $\Phi$ is a type $A$ root system. It was proved in \cite{Shareshian-Wachs2016} that if $\gG=\gG(I)$ for an ideal $I \subseteq \Phi^+$ then $X_\gG(\vx;t)$ is a symmetric function, that is, if we write 
\begin{equation} \label{symco}
X_{\gG(I)}(\vx;t)=\sum_{j=0}^{|I|}f_{j,I}(\vx)t^j,
\end{equation}
then each $f_{j,I}(\vx)$ is a symmetric function.

In \cite{Brosnan-Chow2018}, Brosnan and Chow proved that, as conjectured in \cite{Shareshian-Wachs2016}, the Frobenius characteristic of the dot action representation on $H^\ast(\Hess(s,H(I)))$ is, up to a sign twist,  the chromatic quasisymmetric function $X_{\gG(I)}(\vx;t)$.  The cohomology of $\Hess(s,H(I))$ is concentrated in even degrees.  Let $\sgn$ be the sign representation of $S_n$. The Brosnan--Chow Theorem~\cite[Theorem 129]{Brosnan-Chow2018} says that, with $f_{j,I}$ as in~\eqref{symco},
$$
\fch(H^{2j}(\Hess(s,H(I))) \otimes \sgn)=f_{j,I}(\vx)
$$
for all $j$.  Equivalently,
\begin{equation} \label{brosnanchow}
\sum_{j \geq 0} \fch\left(H^{2j}(\Hess(s,H(I)))\right)t^j=\omega\left(X_{\gG(I)}(\vx;t)\right).
\end{equation}
With the Brosnan--Chow Theorem in hand, the above-mentioned work of Athanasiadis and of the third author and Wachs allow us to determine, respectively, 
the irreducible decomposition and the character values of the dot action in type $A$.  

In contrast, very little is known about the dot action representation for irreducible root systems of types other than $A$.  One might attribute this to the lack of an analogue of the chromatic quasisymmetric function in such types.  An analysis of part of the proof of the Brosnan--Chow Theorem (the original one in \cite{Brosnan-Chow2018}) gives further insight.  A key point is that, given an a priori unknown representation $\phi:S_n \rightarrow GL(V)$, one can determine $\phi$ by determining the dimensions $\dim V^{S_\lambda}$ for every partition $\lambda\in \pn(n)$, where $\pn(n)$ denotes the partitions of $n$.  Here $S_\lambda \leq S_n$ is a direct product of symmetric groups whose orbits on $[n]$ have the parts of $\lambda$ as their lengths, and $V^{S_\lambda}$ is the $S_\lambda$-invariant subspace of $V$. 

There are two ways one might think about this fact (which are essentially equivalent).  The first is that  $\dim V^{S_\lambda}$ gives the multiplicity of the monomial symmetric function $m_\lambda$ in the Frobenius characteristic of $\phi$, and every symmetric function of degree $n$ has a unique expansion in terms of these $m_\lambda$.  The second is that no value $\dim V^{S_\lambda}$ can be derived from the other such values if one does not know $\phi$.  This means that the set $F:=\{\dim V^{S_\lambda}:\lambda \in \pn(n)\}$ provides $|\pn(n)|$ independent pieces of information about $\phi$. Since the dimension of the space of class functions on $S_n$ is also $|\pn(n)|$ and $\phi$ is determined by its character, $\phi$ is determined by $F$.

The set $\{S_\lambda: \lambda \in \pn(n)\}$ is a set of representatives for the $S_n$-conjugacy classes of parabolic subgroups of $W\cong S_n$.  If the irreducible root system $\Phi$ is not of type $A$, then the number of conjugacy classes of parabolic subgroups of the associated Weyl group $W$ is less than the number of conjugacy classes of $W$.  Therefore, one cannot determine a representation $\phi:W \rightarrow GL(V)$ by computing the dimensions $\dim V^{W_J}$ for a set $\{W_J\}$ of representatives of the conjugacy classes of parabolic subgroups of $W$.  This fact explains some of the difficulty in determining dot action representations in types other than A.  It limits what we can prove here.

\subsection{More definitions and the main general result} Given a subset $J \subseteq \Delta$, let $P_J \leq G$ denote the corresponding parabolic subgroup and $W_J\leq W$ its associated Weyl group, where $s_\alpha$ is the simple reflection determined by $\alpha$.  Set $\Phi_J:=\Z J \cap \Phi$.  The key facts for us are as follows.
\begin{itemize}
\item The group $P_J$ is a semidirect product, $P_J=L_JU_J$.  The normal subgroup $U_J$ is the unipotent radical of $P_J$ and the complement $L_J$ is a \emph{Levi subgroup}.
\item The group $L_J$ is reductive and contains $T$.
\item The group $B \cap L_J$ is a Borel subgroup of $L_J$.
\item The Lie algebra $\fl_J$ of $L_J$ admits a Cartan decomposition
$$
\fl_J=\ft \oplus \bigoplus_{\alpha \in \Phi_J} \fg_\alpha.
$$
\item The group $W_J= \langle s_\alpha :\alpha \in J \rangle $ is generated by the simple reflections $s_\alpha$ corresponding to $\alpha\in J$.
\item The Weyl group of $L_J$ is $N_{L_J}(T)/T \cong W_J$
\item Each right coset of $W_J$ in $W$ contains a unique element of minimal length with respect to the usual length function on $W$.
\end{itemize}

We are interested in the restriction of a dot action representation of $W$ to $W_J$.  Given $v \in W$, we write $\dot{v}$ for a representative element of $N_G(T)$ mapped to $v$ by the standard projection to $W$. We let 
\[
{^J}W:= \{  v\in W : \ell(s_{\alpha} v) = \ell(v)+1 \, \text{ for all } \, \alpha\in J \}
\]
denote the set of minimal length right coset representatives of $W_J$ in $W$.  We need the following result of the second author~\cite[Proposition 5.2]{Precup2013}.

\begin{lemma} \label{martha2}
If $H \subseteq \fg$ is a Hessenberg space and $v \in {{^J}W}$, then $H_v\coloneqq \Ad(\dot{v})H \cap \fl_J$ is a Hessenberg space in the Levi subalgebra $\fl_J$ with respect to the Borel subalgebra $\fb \cap \fl_J$.
\end{lemma}

Given a Hessenberg space $H$ and $w\in W$, we define the \emph{Hessenberg length} of $w$ to be
\[
\ell_H(w)\coloneqq \left|\{\alpha \in \Phi_H^- : w(\alpha) \in \Phi^+\}\right|.
\]
Note that when $H=\fg$ we get $\ell_H(w)=\ell(w)$ is the usual Coxeter length of $w$. When $H=\fb$, we have $\ell_H(w) = 0$ for all $w\in W$.

Given a regular semisimple $s \in \ft$ and any subset $J \subseteq \Delta$, we observe that $s$ is also a regular semisimple element of $\fl_J$.  So, given $v \in {^J}W$ we can consider the regular semisimple Hessenberg variety $\Hess(s,H_v) \subseteq L_J/(B \cap L_J)$ by Lemma \ref{martha2}.  The cohomology of $\Hess(s,H_v)$ admits a dot action representation of $W_J$.

\begin{conjecture} \label{maincon}
If $s \in \ft$ is regular semisimple, $H \subseteq \fg$ is a Hessenberg space, and $J \subseteq \Delta$, then for each $j \geq 0$, we have an isomorphism of $W_J$-representations (with respect to the various dot actions),
$$
H^{2j}(\Hess(s,H))\downarrow^W_{W_J} \cong \bigoplus_{v \in {^J}W}H^{2(j-\ell_H(v))}(\Hess(s,H_v)).
$$
Equivalently,
$$
\sum_{j \geq 0}H^{2j}(\Hess(s,H))\downarrow^W_{W_J}t^j \cong \sum_{v \in {^J}W}t^{\ell_H(v)}\sum_{k \geq 0}H^{2k}(\Hess(s,H_v))t^k
$$
as graded $W_J$-representations.
\end{conjecture}

Our main result proves the conjecture in a special case.

\begin{thm} \label{mainthm}
The claim in Conjecture \ref{maincon} is correct when every irreducible component of $\Phi_J$ has type $A$.
\end{thm}

One reason the dot action of the symmetric group has received considerable attention is the previously long-standing conjecture of Stanley and Stembridge that the associated chromatic quasisymmetric function has a positive expansion in the basis of elementary symmetric functions~\cite[Conjecture 5.5]{Stanley-Stembridge1993}. (The conjecture as stated in~\cite{Stanley-Stembridge1993} is more general, but equivalent to this formulation by work of Guay-Paquet in \cite{Guay-Paquet2013}.) In \cite{Hikita2024}, Hikita proves the Stanley--Stembridge Conjecture; in the language of the dot action, this proves the dot action representation is a direct sum of permutation representations whose pointwise stabilizers are parabolic subgroups, that is, induced representations of the form $\mathbf{1}_{W_J} \uparrow_{W_J}^{S_n}$. Accordingly, we obtain the following.

\begin{cor} \label{hpos}
Let $\Hess(s,H)$ be any regular semisimple Hessenberg variety in the flag variety $G/B$ and let $W$ be the Weyl group of $G$.  If $W_J \leq W$ is a parabolic subgroup that is isomorphic to a symmetric group, then the restriction of the dot action of $W$ on the total cohomology $H^\ast(\Hess(s,H))$ to $W_J$ is the direct sum of representations arising from permutation actions on coset spaces of parabolic subgroups of $W_J$.
\end{cor}

In \cite[Conjecture 5.1]{Shareshian-Wachs2016}, Shareshian and Wachs propose a refined version of the Stanley--Stembridge Conjecture.  If this refined version holds, the next conjecture follows.

\begin{conjecture} \label{hposq}
Let $\Hess(s,H)$ be any regular semisimple Hessenberg variety in flag variety $G/B$ and let $W$ be the Weyl group of $G$.  If $W_J \leq W$ is a parabolic subgroup that is isomorphic to a symmetric group, then for each $j \geq 0$ the restriction of the dot action of $W$ on $H^{2j}(\Hess(s,H))$ to $W_J$ is the direct sum of representations arising from permutation actions on coset spaces of parabolic subgroups of $W_J$.
\end{conjecture}

\subsection{Structure of the paper} In Section~\ref{mainproof}, we prove Theorem~\ref{mainthm} for arbitrary Weyl groups. In Section~\ref{bcdproof}, we  focus on types $B_n$, $C_n$ and $D_n$ and prove Theorem~\ref{bcd}. We introduce a graph associated to each ideal $I$ in one of these classical root systems and give an explicit combinatorial formula--in terms of chromatic symmetric functions--for the restriction $H^{2j}(\Hess(s,H))\downarrow^W_{W_J}$ using this graph. Finally, in Section~\ref{sec.Lesnevich} we study the dot action representations of types $B_n$ and $C_n$ corresponding to the same ideal under the natural poset isomorphism, proving a conjecture of Lesnevich.

\subsection*{Acknowledgments} The authors are grateful to Patrick Brosnan and Nathan Lesnevich for helpful conversations. The first author is supported by an AMS--Simons Travel Grant and the second author by NSF CAREER grant DMS-2237057. AI Use Disclosure: The authors used an artificial intelligence tool to assist with proof reading and formatting figures. The authors reviewed and verified all changes and remain fully responsible for the content, mathematical correctness, and final wording of the manuscript.

%%%%%%%%%%

\section{The proof of Theorem \ref{mainthm}} \label{mainproof}

Our proof of Theorem~\ref{mainthm} leverages the structure of regular Hessenberg varieties. Specifically, the Betti numbers of these varieties yield the dimensions of the invariant spaces of the dot action representation on $H^*(\Hess(s,H))$ under the subgroups $W_J$. We can therefore use these Betti numbers to compute the necessary restrictions. 

Given a subset $K\subseteq \Delta$, set $n_K:=\sum_{\alpha\in K}E_\alpha$ where $E_\alpha\in \fg_{\alpha}$ is a fixed root vector and let $s_K\in \ft$ such that $\alpha(s_K)=0$ if and only if $\alpha\in K$. Then $\fl_K=\fz(s_K)$ is a standard Levi subgroup and 
\begin{equation}\label{eqn.reg.elt}
x_K:= n_K+s_K
\end{equation}
is a regular element of $\fg$, since $n_K$ is a regular nilpotent element of $\fl_K$. We call $\Hess(x_K,H)$ a \emph{regular Hessenberg variety}.  The following lemma follows directly from work of the second author~\cite[Theorem 5.4]{Precup2013}, see also~\cite[Lemma 2]{Precup2018}.

\begin{lemma}\label{lem.Betti} Let $K\subseteq \Delta$ and $x_K$ be the associated regular element. Let $H\subseteq \fg$ be a Hessenberg space. The regular Hessenberg variety $\Hess(x_K,H)$ has an affine paving obtained by intersecting $\Hess(x_K, H)$ with the Schubert cells $C_w\coloneqq B\dot wB/B$ for $w\in W$. Moreover, $C_w\cap \Hess(x_K,H)\neq \emptyset$ if and only if $w^{-1}(K) \subseteq \Phi_H$ and, if nonempty,
\[
\dim C_w\cap \Hess(x_K, H) = \ell_H(w).
\] 
Consequently, the Poincar\'e polynomial of $\Hess(x_K,H)$ is 
\[
\Poin(\Hess(x_K,H);t) = \sum_{\substack{w\in W\\ w^{-1}(K)\subseteq \Phi_H}} t^{2\ell_H(w)}.
\]
\end{lemma}

For our purpose, we define the \emph{inversion set} of $w$ by $\Inv(w)\coloneqq\{\gamma\in \Phi^- : w(\gamma)\in\Phi^+\}$. Then the (Coxeter) length of $w$ is $\ell(w)=|\Inv(w)|$ and $\ell_H(w) = |\Inv(w)\cap \Phi_H^-|$.
Recall the standard fact that, for a fixed subset $J\subseteq\Delta$, every $w\in W$ has a unique $y\in W_J$ and a unique $v\in {^J}W$ such that $\ell(w)=\ell(y)+\ell(v)$.
The following is a slight generalization of~\cite[Lemma 8 and Equation (7)]{Precup2018}.

\begin{lemma}\label{lem.Betti.decomp} 
Let $K\subseteq J\subseteq\Delta$, and $w\in W$ with the decomposition $w=yv$ for $y\in W_J$, $v\in {^J}W$. Let $H\subseteq \fg$ be a Hessenberg space. Then $w^{-1}(K)\subseteq \Phi_H$ if and only if $y^{-1}(K)\subseteq \Phi_{H_v}$ where $H_v$ is the corresponding Hessenberg space in $\fl_J$ defined in Lemma \ref{martha2}. Furthermore, in this case, one has $\ell_H(w)=\ell_H(v)+\ell_{H_v}(y)$.
\end{lemma}
\begin{proof}
Since $K\subseteq J$ and $y\in W_J$, we have $y^{-1}(K)\subseteq\Phi_J$. Then
\[
w^{-1}(K) \subseteq \Phi_H \Leftrightarrow y^{-1}(K)\subseteq v(\Phi_H) \Leftrightarrow y^{-1}(K)\subseteq v(\Phi_H) \cap \Phi_J.
\]
Since $v(\Phi_H)\cap \Phi_J = \Phi_{H_v}$ by definition, the first part of the lemma follows. 

For the second part, as $\ell(w)=\ell(y)+\ell(v)$, we have
\[
\Inv(w)= v^{-1}\Inv(y)\sqcup\Inv(v) \quad\text{(disjoint union)}. 
\]
Taking the intersection with $\Phi_H^-$ gives
\[
\ell_H(w)=|v^{-1}\Inv(y)\cap \Phi_H^-|+\ell_H(v).
\]
It suffices to show that $|v^{-1}\Inv(y)\cap \Phi_H^-|=\ell_{H_v}(y)$.
Since $y\in W_J$, we have $\Inv(y)\subseteq \Phi_J^-$, hence
\[
|v^{-1}\Inv(y)\cap \Phi_H^-|=|\Inv(y)\cap v(\Phi_H^-)|=|\Inv(y)\cap v(\Phi_H^-)\cap \Phi_J^-|.
\]
The lemma follows provided that we prove $v(\Phi_H^-) \cap \Phi_J^- = \Phi_{H_v}^-$, that is, we must show $v(\Phi_H^-) \cap \Phi_J^- = v(\Phi_H)\cap \Phi_J^-$. Suppose $\gamma\in \Phi_H$ such that $v(\gamma)\in \Phi_J^-$. Since $v\in {{^J}W}$ we have $\gamma\in v^{-1}(\Phi_J^-) \subseteq \Phi^-$. Thus $\gamma\in \Phi_H^-$, and the proof is complete.
\end{proof}

\begin{thm} \label{thm: PoincarePoly}
Let $K\subseteq J\subseteq \Delta$, and $H$ be a Hessenberg space. The Poincar\'e polynomial of the regular Hessenberg variety $\Hess(x_K, H)$ decomposes,
\[
\Poin(\Hess(x_K, H); t) = \sum_{v\in {^J}W} t^{2\ell_H(v)} \Poin(\Hess(x_K, H_v); t)
\]
where $\Hess(x_K, H_v)$ is the regular Hessenberg variety associated to $x_K\in \fl_J$ in the flag variety $L_J/B_J$ of the Levi subgroup $L_J$.    
\end{thm}

\begin{proof}
Applying Lemma \ref{lem.Betti} and Lemma \ref{lem.Betti.decomp}, we have 
\begin{align*}
\Poin(\Hess(x_K,H);t)
    &=\sum_{\substack{w\in W \\ w^{-1}(K)\subseteq\Phi_H}} t^{2\ell_H(w)}
    =\sum_{v\in {^J}W}t^{2\ell_H(v)}\left(
    \sum_{\substack{y\in W_J\\ y^{-1}(K)\subseteq\Phi_{H_v}}} t^{2\ell_{H_v}(y)}\right)\\
    &= \sum_{v\in {^J}W} t^{2\ell_H(v)} \Poin(\Hess(x_K, H_v); t). \qedhere
\end{align*}
\end{proof}

Recall the following result of B\u{a}libanu--Crooks \cite[Proposition 4.7]{Balibanu-Crooks2024}. A version of this statement was first proved in type $A$ by Brosnan--Chow in \cite[Theorem 127]{Brosnan-Chow2018}. In the special case where $x_K$ is replaced by a regular nilpotent element, the result was proved by Abe--Horiguchi--Masuda--Murai--Sato in~\cite{Abe-Horiguchi-Hessenberg2020} using hyperplane arrangements.

\begin{prop} \label{prop:FixPoints}
Let $x_K$ be a regular element associated with $K\subseteq\Delta$ as in~\eqref{eqn.reg.elt} and $s\in \ft$ a regular semisimple element. Then
\[
    H^*(\Hess(x_K,H))\cong H^*(\Hess(s,H))^{W_K},
\]
where $H^*(\Hess(s,H))^{W_K}$ is the $W_K$-invariant subspace of the dot action representation on $H^*(\Hess(s,H))$.
\end{prop}

To prove Theorem~\ref{mainthm}, we express representations of a Young subgroup $S_{\lambda}$ in terms of symmetric functions as in the case of representations of $S_n$.
For a more detailed introduction to symmetric functions and representations of $S_n$, we refer the reader to \cite{SaganSymRep} and \cite{StanleyEC2}.

Let $\{V^{\lambda}: \lambda\in \pn(n) \}$ denote the collection of all simple $\mathbb{C}[S_n]$-modules.  
Then 
\[
\left\{V^{\lambda}\otimes V^{\mu} : \lambda\in \pn(n) ,\, \mu\in \pn(m) \right\}
\]
consists of all simple $\C[S_n\times S_m]$-modules under the diagonal action.
Let $\langle\cdot,\cdot\rangle$ be the usual inner product on the ring of virtual characters $\mathrm{R}_{\C}(G)$ of $G$, where $G$ is either $S_n$ or $S_n\times S_m$. 
Then the irreducible characters of $G$ form an orthonormal basis.
 
Let $\Lambda=\bigoplus_{n\ge 0}\Lambda_n$ be the algebra of symmetric functions over $\mathbb{C}(t)$. 
The subspace $\Lambda_n$ of degree $n$ is equipped with the \emph{Hall inner product} $\langle \cdot, \cdot\rangle_{\Lambda_n}$, under which the set of Schur functions $\{s_{\lambda}: \lambda\in \pn(n)\}$ is an orthonormal basis, and monomial symmetric functions $\{m_\lambda: \lambda\in \pn(n)\}$ and homogeneous symmetric functions $\{h_\lambda: \lambda\in \pn(n)\}$ are dual bases.
Consider the algebra of doubly symmetric functions $\Lambda\otimes\Lambda$, each of its bigraded piece $\Lambda_n\otimes\Lambda_m$ carries a natural inner product induced from those of $\Lambda_n$ and $\Lambda_m$:
\[
    \langle f_1\otimes f_2, g_1\otimes g_2\rangle_{\Lambda_n\otimes\Lambda_m}\coloneqq \langle f_1,g_1\rangle_{\Lambda_n}\cdot \langle f_2,g_2\rangle_{\Lambda_m} \quad \text{ for all } \,f_1,g_1\in\Lambda_n, f_2,g_2\in\Lambda_m.
\]
Given a $\mathbb{C}[S_n\times S_m]$-module $V=\bigoplus_{\substack{\lambda\in \pn(n)\\ \mu\in \pn(m)}}c_{\lambda,\mu} V^{\lambda}\otimes V^{\mu}$, we define its Frobenius image by
\[
    \ch(V)\coloneqq \sum_{\substack{\lambda\in \pn(n)\\ \mu\in \pn(m)}}c_{\lambda,\mu} s_{\lambda}\otimes s_{\mu} \in \Lambda \otimes \Lambda.
\]
A standard fact states that the Frobenius characteristic map $\ch: \mathrm{R}_{\C}(S_n)\rightarrow \Lambda_n$ is an isometry, that is for any $\C[S_n]$-modules $V$ and $W$,
\[
    \langle V, W\rangle=\langle \ch(V), \ch(W)\rangle_{\Lambda_n}.
\]
This fact extends naturally to $\ch: \mathrm{R}_{\C}(S_n\times S_m)\rightarrow \Lambda_n\otimes\Lambda_m$.

\begin{proof}[Proof of Theorem \ref{mainthm}]

To keep the notation simple, assume that $W_J\cong S_{a}\times S_{b}$ with $a+b=n$.
The case where $W_J\cong S_\lambda$ for $\lambda$ a partition with more than two parts is analogous. 

Let $K\subseteq J$, so $W_K\simeq S_\mu\times S_\nu$ for a partition $\mu$ of $a$ and a partition $\nu$ of $b$. 
By Theorem \ref{thm: PoincarePoly}, we have 
\begin{equation}
\sum_{j\ge 0}\dim\left(H^{2j}(\Hess(x_K,H))\right)t^{j}=\sum_{v\in {^J}W}t^{\ell_{H}(v)}\sum_{k\ge 0}\dim\left(H^{2k}(\Hess(x_K,H_v))\right)t^k.  \label{eq: Poincare}
\end{equation}
Using Proposition \ref{prop:FixPoints} and the fact that $W_K\leq W_J$ we have,
\begin{equation} \label{eqn.LHS.Poincare1}
\begin{split}
\dim H^{2j}(\Hess(x_K,H))
    &=\dim H^{2j}(\Hess(s,H))^{W_K} \\  
    &=\left\langle H^{2j}(\Hess(s,H))\downarrow^{W}_{W_K}, \mathbf{1}_{W_K}\right\rangle \\
    &=\left\langle H^{2j}(\Hess(s,H))\downarrow^{W}_{W_J}\downarrow^{W_J}_{W_K}, \mathbf{1}_{W_K}\right\rangle, 
\end{split}
\end{equation}
where $\langle\cdot,\cdot\rangle$ is the usual inner product on the ring of virtual characters. 
By Frobenius reciprocity, 
\begin{equation}\label{eq: InnerProduct}
\begin{split}
\left\langle H^{2j}(\Hess(s,H))\downarrow^{W}_{W_J}\downarrow^{W_J}_{W_K}, \mathbf{1}_{W_K}\right\rangle
&=\left\langle H^{2j}(\Hess(s,H))\downarrow^W_{W_J},\mathbf{1}_{W_K}\uparrow^{W_J}_{W_K}\right\rangle \\
&=\left\langle H^{2j}(\Hess(s,H))\downarrow^W_{W_J}, \left(\mathbf{1}_{S_{\mu}}\otimes \mathbf{1}_{S_{\nu}}\right)\uparrow^{S_a\times S_b}_{S_{\mu}\times S_{\nu}} \right\rangle  \\
&=\left\langle H^{2j}(\Hess(s,H))\downarrow^W_{W_J},  (\mathbf{1}_{S_\mu}) \uparrow_{S_\mu}^{S_a} \otimes (\mathbf{1}_{S_{\nu}}) \uparrow_{S_\nu}^{S_b} \right\rangle. 
\end{split}
\end{equation}
Passing to the Hall inner product on $\Lambda_a\otimes\Lambda_b$,
since  $\ch(\mathbf{1}_{S_\lambda}\uparrow_{S_\lambda}^{S_n})=h_{\lambda} $ for any $\lambda\in \pn(n)$, we combine~\eqref{eqn.LHS.Poincare1} and~\eqref{eq: InnerProduct} to obtain: 
\[
    \dim H^{2j}(\Hess(x_K,H))
    =\left\langle \ch\left(H^{2j}(\Hess(s,H)\downarrow^W_{W_J}\right),h_{\mu}\otimes h_{\nu}\right\rangle_{\Lambda_a\otimes\Lambda_b}.
\]
In other words, $\dim H^{2j}(\Hess(x_K,H))$ is equal to the coefficient of $m_{\mu}\otimes m_{\nu}$ in $\ch\left(H^{2j}(\Hess(s,H)\downarrow^W_{W_J}\right)$.

Using Proposition~\ref{prop:FixPoints} once more we also have 
\begin{align*}
    \dim\left(H^{2k}(\Hess(x_K,H_v))\right)
    &= \dim\ H^{2k}(\Hess(s,H_v))^{W_K} \\
    &=\left\langle \ch\left(H^{2k}(\Hess(s,H_v))\right), h_{\mu}\otimes h_{\nu} \right\rangle_{\Lambda_a\otimes\Lambda_b}.
\end{align*}

Hence \eqref{eq: Poincare} can be rephrased as
\begin{align*}
&\sum_{j\ge 0} \left\langle \ch\left(H^{2j}(\Hess(s,H))\downarrow^W_{W_J}\right), h_{\mu}\otimes h_{\nu}\right\rangle_{\Lambda_a\otimes\Lambda_b}\, t^j \\ & \qquad\qquad 
= \sum_{v\in{^J}W}t^{\ell_H(v)} \sum_{k\ge 0} \left\langle \ch\left(H^{2k}(\Hess(s,H_v))\right), h_{\mu}\otimes h_{\nu}\right\rangle_{\Lambda_a\otimes\Lambda_b} \, t^k
\end{align*}
for all $\mu\in\pn(a)$ and $\nu\in\pn(b)$.
In other words, the two doubly symmetric functions in $\Lambda_a\otimes\Lambda_b$ have the same coefficients when expanding in terms of the basis $\{m_{\mu}\otimes m_{\nu}\}$. 
Therefore, we must have
\[
\sum_{j\ge 0}\ch\left(H^{2j}(\Hess(s,H))\downarrow^W_{W_J}\right)t^j=\sum_{v\in{^J}W}t^{\ell_H(v)}\sum_{k\ge 0}\ch\left(H^{2k}(\Hess(s,H_v))\right)t^k. \qedhere
\]
\end{proof}

\section{Proof and discussion of Theorem~\ref{bcd}} \label{bcdproof}

Recall that the Brosnan--Chow Theorem, stated in~\eqref{brosnanchow}, tells us the Frobenius characteristic of the type $A$ dot action representation is a chromatic quasisymmetric function, up to sign twist.  By Theorem~\ref{mainthm}, the restriction of the dot action for an arbitrary Weyl group to a parabolic subgroup of type $A$ is, in fact, a sum of type $A$ dot action representations. 
In this section, we develop combinatorics to describe the restriction to a maximal type $A$ parabolic subgroup in greater detail for the classical types $B_n$, $C_n$, and $D_n$. 

Assume $G$ is a reductive group with associated root system $\Phi$ of type $B_n$, $C_n$, or $D_n$.  Throughout this section, let $W_J \cong S_n$ be the maximal parabolic subgroup of the Weyl group $W$ of $G$ associated with the subset 
\[
J = \{ e_i-e_{i+1} : 1\leq i \leq n-1 \} \subset \Delta.
\]
The subroot system of $\Phi$ spanned by $J$ is
$$
\Phi_J:=\{\pm(e_i-e_j):1 \leq i<j \leq n\} \subset \Phi.
$$
Recall that $P_J$ is the maximal standard parabolic subgroup of $G$ associated to $\Phi_J$, with Levi subgroup $L_J$. We denote the Lie algebra of $L_J$ by $\fl_J \subseteq \fg$.

Let $H=H(I)$ be a Hessenberg space with respect to the Borel subalgebra $\fb$  in the Lie algebra $\fg$ of $G$. So, $I$ is an order ideal in $\Phi^+$ and
$$
H=\fb \oplus \bigoplus_{\alpha \in I}\fg_{-\alpha}.
$$
Given a shortest coset representative $v\in {^J}W$ recall that, by Lemma~\ref{martha2}, $H_v:=\Ad(\dot{v})H \cap \fl_J$ is a Hessenberg space in $\fl_J$ with respect to $\fb \cap \fl_J$. As observed in the proof of Lemma~\ref{lem.Betti.decomp}, $\Phi_{H_v}^- = v(\Phi_H^-)\cap \Phi_J^-$. Because $H=H(I)$, we have that $\Phi_H^- = -I$. In particular, $H_v \subseteq \fl_J$ is the Hessenberg space associated to the order ideal $v(I)\cap \Phi_J^+$:
\begin{equation} \label{hv}
H_v=(\fb \cap \fl_J) \oplus \bigoplus_{\alpha \in v(I) \cap \Phi_J^+}\fg_{-\alpha}.
\end{equation}

Because $\Phi_J$ is a type $A$ root system, we consider the associated naturally labeled unit interval graph $\Gamma(v(I)\cap \Phi_J^+)$ with vertex set $[n]$ and edges $\{ij : e_i - e_j \in v(I)\}$. We see now that the natural unit interval graph associated to $H_v$ is indeed $\Gamma_{H,v}$ (where $\Gamma_{H,v}$ is the graph defined in Section~\ref{sec.classical}).  We can now prove Theorem~\ref{bcd}.

\begin{proof}[Proof of Theorem~\ref{bcd}] According to the Brosnan--Chow Theorem~\eqref{brosnanchow},
$$
\sum_{j \geq 0} \fch\left(H^{2j}(\Hess(s,H_v))\right)t^j=\omega(X_{\Gamma_{H,v}}(\vx;t)).
$$
Now Theorem~\ref{bcd} follows immediately from Theorem~\ref{mainthm}.
\end{proof}

\begin{rem} If $\Phi$ is of type $B_n$ or type $C_n$ with $n>2$ then $W$ has only one standard parabolic subgroup isomorphic to $S_n$, so $J$ such that $W_J\simeq S_n$ is uniquely determined.  Both standard maximal parabolics in type $B_2$ or $C_2$ are isomorphic to $S_2$, and these two standard parabolics are not conjugate in $W$.  Which representation of $S_2$ we obtain by restriction of a given dot action can depend on the chosen maximal parabolic.  If $\Phi$ is of type $D_n$ with $n \geq 4$ then $W$ has two standard parabolic subgroups isomorphic to $S_n$, the group $W_J$ mentioned in Theorem~\ref{bcd} and the one obtained by replacing $\alpha_{n-1}$ in $J$ with $\alpha_n$.  If $n$ is odd then these two subgroups are conjugate in $W$, and the representations of $S_n$ obtained from dot actions by restriction to either one of them are isomorphic, assuming ``standard" isomorphisms of these subgroups with $S_n$ in its usual guise.  If $n$ is even, the two given parabolic subgroups are not conjugate in $W$ and in fact intersect different $W$-conjugacy classes nontrivially.  In these cases, there are dot action representations whose restrictions to the respective subgroups are not isomorphic. 
\end{rem}

%%%%%%%%%%%%%

\subsection{Signed permutations} 
Before introducing the key combinatorial construction, we recall that the Weyl groups $W(B_n)= W(C_n)$ and $W(D_n)$ can be identified as signed permutation groups.

Let $[\bar n]:= \{\, \bar i :i \in [n]\}$. A \emph{signed permutation} $v$ on $[n]$ is a bijection on $[n]\cup [\bar n]$ satisfying $v(\bar i)=\bar{v(i)}$ for all $i$.
Such a permutation can be denoted using two-line notation as
\[
v=\left[\begin{array}{cccc|cccc}
     1 & 2 & \cdots & n & \bar n & \bar{n-1} & \cdots & \bar 1 \\
     v(1) & v(2) & \ldots & v(n) & v(\bar n) & v(\bar{n-1}) & \cdots & v(\bar1)
    \end{array}
    \right].
\]
We may also write $v=v(1)v(2)\cdots v(n) v(\bar n)\cdots v(\bar 1)$, which is called the \emph{one-line notation} of $v$. 
Since $v$ is completely determined by its values on $[n]$, this notation can be further abbreviated to $v=v(1)v(2)\cdots v(n)$, which is called the \emph{window notation} of $v$. 

For the classical irreducible root systems of types $B_n$, $C_n$, or $D_n$, every element $v$ of the corresponding Weyl group may be regarded as a signed permutation by defining $e_{v(i)}=v(e_i)$, where we adopt the convention $e_{\bar i}=-e_i$ for $i\in [n]\cup [\bar n]$.
Under this identification, the Weyl group $W(B_n)= W(C_n)$ is isomorphic to the group of signed permutations on $[n]$, while $W(D_n)$ is isomorphic to the subgroup of \emph{even signed permutations}, namely those whose window notation contains an even number of barred entries.

Define a total order $\preceq$ on $[n]\cup [\bar n]$ by
\begin{equation} \label{eq.order}
1 \prec 2 \prec 3 \prec \ldots \prec n  \prec \bar n \prec \bar{n-1} \prec \bar{n-2} \prec \ldots \prec \bar{1}.
\end{equation}
It is straightforward to check that $\preceq$ gives us a particularly simple description of certain positive roots: if $a,b\in [n]\cup [\bar n]$ then $e_a-e_b\in \Phi^+$ if and only if $a\prec b$.
Our next result tells us how to determine which signed permutations are elements of ${^J}W$. 

\begin{lemma} \label{lem.shortest}
Suppose $\Phi$ is an irreducible root system of type $B_n$, $C_n$, or $D_n$ with associated Weyl group $W$ and $v\in W$.  The following are equivalent:
\begin{enumerate}
\item $v\in {^J}W$.
\item For all $1\leq i <j \leq n$, $i$ appears to the left of $j$ in the one-line notation for $v$.
\item $v^{-1}(1)\prec v^{-1}(2)\prec \cdots\prec v^{-1}(n)$.
\end{enumerate}
\end{lemma}

\begin{proof} The equivalence of (2) and (3) follows directly from the definition of the one-line notation and the total order $\prec$.

Next we prove (1) is equivalent to (3). It suffices to show $v\in {^J}W$ if and only if $i$ appears to the left of $i+1$ for all $1\le i\le n-1$.
Since $v\in {^J}W$ and $\ell(w)=\ell(w^{-1})$ for any $w\in W$, we have that for any $\alpha\in J$,
\[
    \ell(s_{\alpha}v)=\ell(v)+1 \Leftrightarrow \ell(v^{-1}s_{\alpha})=\ell(v^{-1})+1 \Leftrightarrow v^{-1}(\alpha)\in \Phi^+.
\]
Recall that $J=\{e_i-e_{i+1} : 1\le i\le n-1\}$. We have 
$v^{-1}(e_i-e_{i+1})=e_{v^{-1}(i)}-e_{v^{-1}(i+1)}$ is a positive root for all $1\le i\le n-1$ if and only if 
\[
    v^{-1}(1)\prec v^{-1}(2)\prec \cdots\prec v^{-1}(n),
\]
as desired.
\end{proof}

\begin{example} \label{ex:ParabolicQuotient}
Let $\Phi= B_3$ or $\Phi= C_3$. We have $J=\{\alpha_1=e_1-e_2, \alpha_2=e_2-e_3\}$ and $W_J\cong S_3$.
The following displays the $8$ elements in ${^J}W$, both their reduced word and one-line notations.  

\begin{align*}
e &= 1~2~3~\bar{3}~\bar{2}~\bar{1} &\hspace{5em}  s_3 &= 1~2~\bar{3}~3~\bar{2}~\bar{1} \\
s_3s_2 &= 1~\bar{3}~2~\bar{2}~3~\bar{1} & s_3s_2s_1 &= \bar{3}~1~2~\bar{2}~\bar{1}~3 \\
s_3s_2s_3 &= 1~\bar{3}~\bar{2}~2~3~\bar{1}  & s_3s_2s_3s_1 &= \bar{3}~1~\bar{2}~2~\bar{1}~3 \\
s_3s_2s_3s_1s_2 &= \bar{3}~\bar{2}~1~\bar{1}~2~3  & s_3s_2s_3s_1s_2s_3 &= \bar{3}~\bar{2}~\bar{1}~1~2~3
\end{align*}
Note that each element has $1$ appearing to the left of $2$ and $2$ appearing to the left of $3$ in the one-line notation, as indicated in Lemma~\ref{lem.shortest}.
\end{example}

\subsection{Graph-theoretic description and a combinatorial formula for the irreducible decomposition} Continuing with our assumption that $\Phi$ is an irreducible root system of type $B_n$, $C_n$, or $D_n$, we will define, for each order ideal $I \subseteq \Phi^+$, a graph $\Gamma(I)$ from which one can determine the restriction of the dot action representation to $W_J$ in a simple combinatorial way. The resulting formulas are similar to those used for naturally labeled unit internal graphs to determine the dot action representation in type $A$.

\begin{definition} Let $I\subseteq \Phi^+$ be an order ideal.
The vertex set $V(I)$ of the graph $\Gamma(I)$ is $[n] \cup \{0\} \cup [ \bar n]$.  The edge set $E(I)$ of $\Gamma(I)$ is determined by the following rules.
\begin{itemize}
\item The graph $\Gamma(I)$ is invariant under the permutation that maps $j$ to $\bar j$ for each $j \in V(I)$.
\item For $1 \leq i<j \leq n$, the edge $\{i,j\}$ lies in $E(I)$ if and only if $e_i-e_j \in I$.
\item For $1 \leq i<j \leq n$, the edge $\{i,\bar j\}$ lies in $E(I)$ if and only if $e_i+e_j \in I$.
\item For $i \in [n]$, the edge $\{i,\bar i \}$ lies in $E(I)$ if and only if $2e_i \in I$.
\item For $i \in [n]$, the edge $\{0,i\}$ lies in $E(I)$ if and only if $e_i \in I$.
\end{itemize}
The last item implies that $0$ is an isolated vertex when $\Phi = C_n$ or $\Phi= D_n$, so we omit that vertex from $\Gamma(I)$ in those cases.
\end{definition}

\begin{figure}[htbp]
    \centering
    \begin{subfigure}[c]{0.48\textwidth}
        \centering
        \begin{adjustbox}{max width=\textwidth}
            \begin{tikzpicture}[
                vertex/.style={font=\scriptsize}
            ]
                \def\df{20}
                \def\w{1 cm}

                \coordinate (100) at (0,0);
                \coordinate (010) at (2,0);
                \coordinate (001) at (4,0);
                \coordinate (110) at (1,1);
                \coordinate (011) at (3,1);
                \coordinate (111) at (2,2);
                \coordinate (021) at (4,2);
                \coordinate (121) at (3,3);
                \coordinate (221) at (4,4);
                
                \path[fill=blue!\df, draw=blue!\df, line width=\w,line cap=round, line join=round] 
                (100) to (110) to (010) to (011) to (001);
                \path[fill=blue!\df, draw=blue!\df, line width=\w,line cap=round, line join=round] 
                (011) to (021);

                \node[vertex] (p100) at (100) {$e_1-e_2$};
                \node[vertex] (p010) at (010) {$e_2-e_3$};
                \node[vertex] (p001) at (001) {$2e_3$};
                \node[vertex] (p110) at (110) {$e_1-e_3$};
                \node[vertex] (p011) at (011) {$e_2+e_3$};
                \node[vertex] (p111) at (111) {$e_1+e_3$};
                \node[vertex] (p121) at (121) {$e_1+e_2$};
                \node[vertex] (p021) at (021) {$2e_2$};
                \node[vertex] (p221) at (221) {$2e_1$};
            
                \draw[very thick] (p100)--(p110)--(p111)--(p121)--(p221);
                \draw[very thick] (p110)--(p010)--(p011)--(p021)--(p121);
                \draw[very thick] (p001)--(p011)--(p111);
            \end{tikzpicture}
        \end{adjustbox}
    \end{subfigure}%
    \hfill
    \begin{subfigure}[c]{0.48\textwidth}
        \centering
        \begin{adjustbox}{max width=\textwidth}
            \begin{tikzpicture}[
                vertex/.style={circle, draw, minimum size=0.65 cm, inner sep=0pt},
                every edge/.style={draw, thick}
            ]
                %Vertices
                \node[vertex] (v1) at (0,0) {\tiny $1$};
                \node[vertex] (v2) at (1.2,0) {\tiny $2$};
                \node[vertex] (v3) at (2.4,0) {\tiny $3$};
                \node[vertex] (v4) at (3.6,0) {\tiny $\bar 3$};
                \node[vertex] (v5) at (4.8,0) {\tiny $\bar 2$};
                \node[vertex] (v6) at (6.0,0) {\tiny $\bar 1$};
                
                \draw (v1)--(v2)--(v3)--(v4)--(v5)--(v6);
                \draw (v1) to [out=45, in=135] (v3) (v2) to [out=45, in=135] (v4);
                \draw (v3) to [out=45, in=135] (v5) (v4) to [out=45, in=135] (v6);
                \draw (v2) to [out=70, in=110] (v5);
            \end{tikzpicture}   
        \end{adjustbox}
    \end{subfigure}
    \caption{An ideal $I$ in $\Phi=C_3$ and corresponding graph $\Gamma(I)$.}
    \label{fig:C3}
\end{figure}
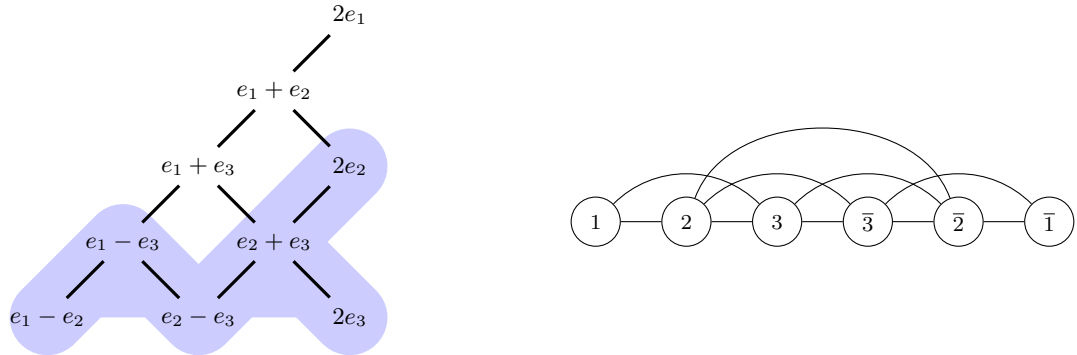

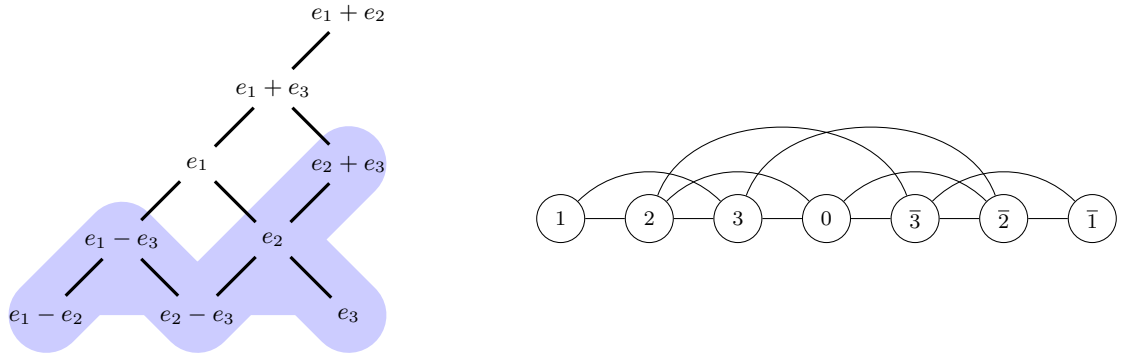
\begin{figure}[htbp]
    \centering
    \begin{subfigure}[c]{0.48\textwidth}
        \centering
        \begin{adjustbox}{max width=\textwidth}
            \begin{tikzpicture}[
                vertex/.style={font=\scriptsize}
            ]
                \def\df{9}
                \def\w{1 cm}

                \coordinate (100) at (0,0);
                \coordinate (010) at (2,0);
                \coordinate (001) at (4,0);
                \coordinate (110) at (1,1);
                \coordinate (011) at (3,1);
                \coordinate (111) at (2,2);
                \coordinate (021) at (4,2);
                \coordinate (121) at (3,3);
                \coordinate (221) at (4,4);
                
                \path[fill=blue!20, draw=blue!20, line width=\w,line cap=round, line join=round] 
                (100) to (110) to (010) to (011) to (001);
                \path[fill=blue!20, draw=blue!20, line width=\w,line cap=round, line join=round] 
                (011) to (021);

                \node[vertex] (p100) at (100) {$e_1-e_2$};
                \node[vertex] (p010) at (010) {$e_2-e_3$};
                \node[vertex] (p001) at (001) {$e_3$};
                \node[vertex] (p110) at (110) {$e_1-e_3$};
                \node[vertex] (p011) at (011) {$e_2$};
                \node[vertex] (p111) at (111) {$e_1$};
                \node[vertex] (p121) at (121) {$e_1+e_3$};
                \node[vertex] (p021) at (021) {$e_2+e_3$};
                \node[vertex] (p221) at (221) {$e_1+e_2$};
            
                \draw[very thick] (p100)--(p110)--(p111)--(p121)--(p221);
                \draw[very thick] (p110)--(p010)--(p011)--(p021)--(p121);
                \draw[very thick] (p001)--(p011)--(p111);
            \end{tikzpicture}
        \end{adjustbox}
    \end{subfigure}%
    \hfill
    \begin{subfigure}[c]{0.48\textwidth}
        \centering
        \begin{adjustbox}{max width=\textwidth}
            \begin{tikzpicture}[
                vertex/.style={circle, draw, minimum size=0.65 cm, inner sep=0pt},
                every edge/.style={draw, thick}
            ]
                %Vertices
                \node[vertex] (v1) at (0,0) {\tiny $1$};
                \node[vertex] (v2) at (1.2,0) {\tiny $2$};
                \node[vertex] (v3) at (2.4,0) {\tiny $3$};
                \node[vertex] (v7) at (3.6,0) {\tiny $0$};
                \node[vertex] (v4) at (4.8,0) {\tiny $\bar 3$};
                \node[vertex] (v5) at (6.0,0) {\tiny $\bar 2$};
                \node[vertex] (v6) at (7.2,0) {\tiny $\bar 1$};
                
                \draw (v1)--(v2)--(v3)--(v7)--(v4)--(v5)--(v6);
                \draw (v1) to [out=45, in=135] (v3) (v4) to [out=45, in=135] (v6);
                \draw (v2) to [out=45, in=135] (v7) (v7) to [out=45, in=135] (v5);
                \draw (v2) to [out=70, in=110] (v4) (v3) to [out=70, in=110] (v5);
            \end{tikzpicture}   
        \end{adjustbox}
    \end{subfigure}
    \caption{An ideal $I$ in $\Phi=B_3$ and corresponding graph $\Gamma(I)$.}
    \label{fig:B3} 
\end{figure}

\begin{figure}[htbp]
    \centering
    \begin{subfigure}[c]{0.48\textwidth}
        \centering
        \begin{adjustbox}{max width=\textwidth}
            \begin{tikzpicture}[
                vertex/.style={font=\scriptsize}
            ]
                \def\df{9}
                \def\w{1 cm}

                \coordinate (1000) at (0,0);
                \coordinate (0100) at (2,0);
                \coordinate (0010) at (4,0);
                \coordinate (0001) at (6,0);
                \coordinate (1100) at (1,1);
                \coordinate (0110) at (3,1);
                \coordinate (0101) at (5,1);
                \coordinate (1110) at (2,2);
                \coordinate (1101) at (4,2);
                \coordinate (0111) at (6,2);
                \coordinate (1111) at (5,3);
                \coordinate (1211) at (6,4);
                
                \path[fill=blue!20, draw=blue!20, line width=\w,line cap=round, line join=round] 
                (1000) to (1100) to (0100) to (0110) to (0010) (0110) to (0111);
                \path[fill=blue!20, draw=blue!20, line width=\w,line cap=round, line join=round] 
                (0100) to (0101) to (0001) (0101) to (0111);

                \node[vertex] (p1000) at (1000) {$e_1-e_2$};
                \node[vertex] (p0100) at (0100) {$e_2-e_3$};
                \node[vertex] (p0010) at (0010) {$e_3-e_4$};
                \node[vertex] (p0001) at (0001) {$e_3+e_4$};
                \node[vertex] (p1100) at (1100) {$e_1-e_3$};
                \node[vertex] (p0110) at (0110) {$e_2-e_4$};
                \node[vertex] (p0101) at (0101) {$e_2+e_4$};
                \node[vertex] (p1110) at (1110) {$e_1-e_4$};
                \node[vertex] (p1101) at (1101) {$e_1+e_4$};
                \node[vertex] (p0111) at (0111) {$e_2+e_3$};
                \node[vertex] (p1111) at (1111) {$e_1+e_3$}; 
                \node[vertex] (p1211) at (1211) {$e_1+e_2$};

                \draw[very thick] (p1000)--(p1100)--(p1110)--(p0110)--(p0010);
                \draw[very thick] (p1100)--(p0100)--(p0110);
                \draw[very thick] (p1100)--(p1101)--(p1111)--(p1211);
                \draw[very thick] (p1110)--(p1111)--(p0111);
                \draw[very thick] (p0100)--(p0101)--(p0111);
                \draw[very thick] (p1101)--(p0101)--(p0001);
                \draw[very thick] (p0110)--(p0111);
            \end{tikzpicture}
        \end{adjustbox}
    \end{subfigure}%
    \hfill
    \begin{subfigure}[c]{0.48\textwidth}
        \centering
        \begin{adjustbox}{max width=\textwidth}
            \begin{tikzpicture}[
                vertex/.style={circle, draw, minimum size=0.65 cm, inner sep=0pt},
                every edge/.style={draw, thick}]
                %Vertices
                \node[vertex] (v1) at (0,0) {\tiny $1$};
                \node[vertex] (v2) at (1.2,0) {\tiny $2$};
                \node[vertex] (v3) at (2.4,0) {\tiny $3$};
                \node[vertex] (v4) at (3.6,0) {\tiny $4$};
                \node[vertex] (v5) at (4.8,0) {\tiny $\bar4$};
                \node[vertex] (v6) at (6.0,0) {\tiny $\bar3$};
                \node[vertex] (v7) at (7.2,0) {\tiny $\bar2$};
                \node[vertex] (v8) at (8.4,0) {\tiny $\bar1$};
                
                \draw (v1)--(v2)--(v3)--(v4) (v5)--(v6)--(v7)--(v8);
                \draw (v3) to [out=45, in=135] (v5) (v4) to [out=45, in=135] (v6);
                \draw (v2) to [out=60, in=120] (v5) (v2) to [out=70, in=110] (v6);
                \draw (v1) to [out=45, in=135] (v3) (v6) to [out=45, in=135] (v8);
                \draw (v2) to [out=45, in=135] (v4) (v5) to [out=45, in=135] (v7);
                \draw (v3) to [out=70, in=110] (v7) (v4) to [out=60, in=120] (v7);
            \end{tikzpicture}
        \end{adjustbox}
    \end{subfigure}
    \caption{An ideal $I$ in $\Phi=D_4$ and corresponding graph $\Gamma(I)$.}
    \label{fig:D4}
\end{figure}
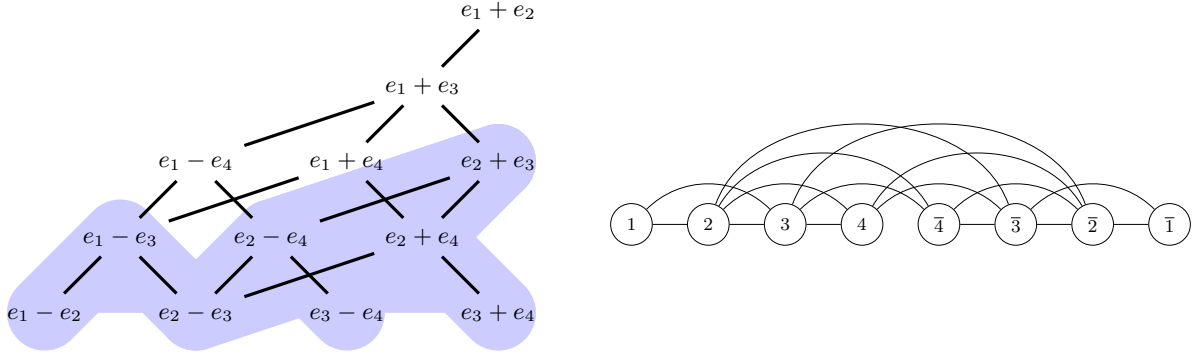

See Figures~\ref{fig:C3}, \ref{fig:B3}, and~\ref{fig:D4} for examples of ideals in root systems of type $B_3$, $C_3$, and $D_4$ and their corresponding graphs.

Recall that $H=H(I)$. Our next lemma tells us how to identify each graph $\Gamma_{H,v}$ for each shortest coset representative $v\in {^J}W$ as an induced subgraph of $\Gamma(I)$, up to relabeling.

\begin{lemma} \label{lem.graphdef} Let $v\in {^J W}$ and $I\subseteq \Phi^+$ an order ideal with associated Hessenberg space $H=H(I)$. Then $\Gamma_{H, v}$ is the induced subgraph of $\Gamma(I)$ on the vertex set $\{ v^{-1}(j) : j\in [n] \}$ after relabeling each vertex $v^{-1}(j)$ by $j$.
\end{lemma}
\begin{proof} Suppose $\gamma \in v(I)\cap \Phi_J^+$. Since $J = \{e_i - e_{i+1}: 1\leq i \leq n-1\}$ and $v\in {^J}W$, there exists $i,j\in [n]$ with $i<j$ such that
\begin{align}\label{eqn.graph.pf}
\gamma = e_i-e_j \in v(I)\cap \Phi_J^+ &\Leftrightarrow e_{v^{-1}(i)} - e_{v^{-1}(j)} \in I  \, \text{ and } \, v^{-1}(i)\prec v^{-1}(j).
\end{align}
Because $v^{-1}(i)\prec v^{-1}(j)$ we know $e_{v^{-1}(i)} - e_{v^{-1}(j)} = e_{a}-e_b$ or $e_{v^{-1}(i)} - e_{v^{-1}(j)} = e_a+e_b$ for some $a,b\in [n]$ with $a<b$. 
Now it's clear from the definition of $\Gamma(I)$ that the RHS of~\eqref{eqn.graph.pf} is equivalent to $\{v^{-1}(i), v^{-1}(j)\}\in E(I)$. 
\end{proof}

For each $v \in {^J}W$, we define $\Gamma(I)_v$ to be the induced subgraph of $\Gamma(I)$ on $\{v^{-1}(j):j \in [n]\}$. 
Next, we aim to compute $\ell_H(v)$ using the graphs $\Gamma(I)_v$ and $\Gamma(I)$. 
Define a total order $\preceq$ on the vertex set $[n] \cup \{0\} \cup [\bar n]$ of $\Gamma(I)$ by
\begin{equation} \label{eq.order2}
1 \prec 2 \prec 3 \prec \ldots \prec n \prec 0 \prec \bar n \prec \bar{n-1} \prec \bar{n-2} \prec \ldots \prec \bar{1}.
\end{equation}
Note that $\preceq$ restricts to~\eqref{eq.order} on $[n]\cup [\bar{n}]$. 
Let $\iota$ be the involution on $[n] \cup \{0\} \cup [\bar n]$ fixing $0$ and exchanging $j$ and $\bar j$ for each $j \in [n]$.  We observe that $\Gamma(I)$ is $\iota$-invariant, thus $\iota$ acts on the edge set $E(I)$ of $\Gamma(I)$.

\begin{cor} \label{cor.hlength}
If $v \in {^J}W$ then $\ell_H(v)$ is equal to the number of $\iota$-orbits on $E(I)$ consisting of edges $\{a,b\}$ in $\Gamma(I)$ such that:
\begin{equation}\label{eqn.edge.count}
a\prec b,\;\;  b\in \{ v^{-1}(j) : j\in [n] \}, \; \text{ and } \;\; a\notin  \{ v^{-1}(j) : j\in [n] \}.
\end{equation}
\end{cor}

\begin{proof} Note that~\eqref{eqn.edge.count} holds for an edge $\{a,b\}$ if and only if it holds for $\{\bar{b}, \bar{a}\}$ so this condition is indeed well-defined on the $\iota$-orbits in $E(\Gamma)$. We prove the corollary by direct inspection, after recalling that $\ell_H(v)$ is the number of $\alpha \in I$ such that $v(\alpha) \in \Phi^-$. 

First, suppose $a=0$ and $\{0,b\}$ is an edge in $\Gamma(I)$ satisfying~\eqref{eqn.edge.count}. This implies $b\in [\bar n]$, $v(b) \in [n]$, and $e_{\bar b}\in I$. Now $\alpha = e_{\bar b}\in I$ such that $v(\alpha) = v(e_{\bar b}) = -e_{v(b)}\in \Phi^-$.  

On the other hand, let $\bar b\in [n]$ such that $e_{\bar b}\in I$ and $v(e_{\bar b})=-e_{v(b)}\in \Phi^-$. This implies $v(b)\in [n]$ and both $\{0, b\}$ and $\{\bar b, 0\}$ are edges of $\Gamma(I)$ satisfying~\eqref{eqn.edge.count}. We conclude there is a bijection between $\iota$-orbit of edges $\{0,b\}$ satisfying~\eqref{eqn.edge.count} and roots $\alpha = e_{\bar b}\in I$ such that $v(\alpha)\in \Phi^-$.

Now suppose $a\neq 0$ and $\{a,b\}\in E(I)$ satisfies~\eqref{eqn.edge.count}. This implies $e_a-e_b\in I$, $v(a)\in [\bar n]$, and $v(b)\in [n]$. Because $v(b)\prec v(a)$, we get that $v(e_a-e_b) = e_{v(a)}-e_{v(b)}\in \Phi^-$. 

On the other hand, let $a,b\in [n]\cup [\bar n]$ such that  $\alpha = e_a-e_b\in I$ and $v(\alpha) = e_{v(a)}-e_{v(b)}\in \Phi^-$. Since $I$ is a subset of positive roots, we have $a\prec b$.  The fact that $e_{v(a)}-e_{v(b)}\in \Phi^-$ implies $v(b)\prec v(a)$. Because $v\in {^J}W$, Lemma~\ref{lem.shortest} now implies $v(b)\in [n]$ and $v(a)\in [\bar n]$ and the edges $\{a,b\}$ and $\{\bar b, \bar a\}$, which might be equal, satisfy~\eqref{eqn.edge.count}. Thus, there is a bijection between $\iota$-orbits of edges $\{a,b\}\in E(I)$ with $a\neq 0$ that satisfy~\eqref{eqn.edge.count} and roots $e_a-e_b\in I$ such that $v(e_a-e_b)\in \Phi^-$. This concludes the proof.
\end{proof}

By Corollary~\ref{cor.hlength}, we can read the value of $\ell_H(v)$ for each $v\in {^J}W$ using the induced subgraph of $\Gamma(I)$ on $\{v^{-1}(j) : j\in [n]\}$. Indeed, suppose we have drawn $\Gamma(I)$ by writing the vertex set in order, left to right, according to $\preceq$. Then given an induced subgraph on the set $\{v^{-1}(j) : j\in [n]\}$, we simply count the number of edges with left endpoint outside the subgraph and with right endpoint in the subgraph, up to symmetry.

\begin{example} \label{exa: subgraphs}
Let $I$ be the order ideal of $\Phi^+(C_3)$ generated by $e_1-e_3$ and $e_2+e_3$.
Consider the signed permutation $v=1~\bar{3}~2~\bar{2}~3~\bar{1}\in{^J}W$. 
The graph below is $\Gamma(I)$, and the induced subgraph $\Gamma(I)_v$ on $\{v^{-1}(1), v^{-1}(2), v^{-1}(3)\}$ is highlighted in blue.
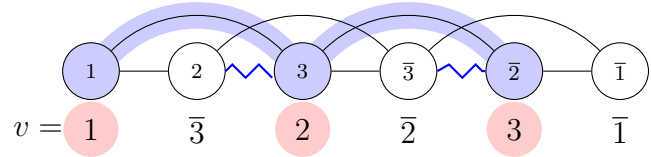
\[
    \begin{tikzpicture}[
                scale=0.7,
                vertex/.style={circle, draw, minimum size=0.75 cm},
                every edge/.style={draw, thick}
            ]
            \coordinate[circle, minimum size=0.75 cm] (p1) at (0,0);
            \coordinate[circle, minimum size=0.75 cm] (p2) at (2,0);
            \coordinate[circle, minimum size=0.75 cm] (p3) at (4,0);
            \coordinate[circle, minimum size=0.75 cm] (p4) at (6,0);
            \coordinate[circle, minimum size=0.75 cm] (p5) at (8,0);
            \coordinate[circle, minimum size=0.75 cm] (p6) at (10,0);

            \path[draw=blue!20, line width=10,line cap=round, line join=round] 
            (p1) to [out=45, in=135] (p3) to [out=45, in=135] (p5);
            
            %Vertices
            \node[vertex, fill=blue!20, label=below:{}] (v1) at (0,0) {\tiny $1$};
            \node[vertex, label=below:{}] (v2) at (2,0) {\tiny$2$};
            \node[vertex, fill=blue!20, label=below:{}] (v3) at (4,0) {\tiny$3$};
            \node[vertex, label=below:{}] (v4) at (6,0) {\tiny$\bar{3}$};
            \node[vertex, fill=blue!20, label=below:{}] (v5) at (8,0) {\tiny$\bar{2}$};
            \node[vertex, label=below:{}] (v6) at (10,0) {\tiny$\bar{1}$};

            \node (w) at (-1,-1.1) {$v=$};
            \node[circle, fill=red!20, minimum size=0.5 cm] (w1) at (0,-1.1) {$1$};
            \node (w2) at (2,-1.1) {$\bar{3}$};
            \node[circle, fill=red!20, minimum size=0.5 cm] (w3) at (4,-1.1) {$2$};
            \node (w4) at (6,-1.1) {$\bar{2}$};
            \node[circle, fill=red!20 , minimum size=0.5 cm] (w5) at (8,-1.1) {$3$};
            \node (w6) at (10,-1.1) {$\bar{1}$};

            \draw (v1)--(v2) (v3)--(v4) (v5)--(v6);
            \draw[thick, decorate, decoration=zigzag, blue!100] (v2)--(v3) (v4)--(v5);
            \draw (v1) to [out=45, in=135] (v3) (v2) to [out=45, in=135] (v4);
            \draw (v3) to [out=45, in=135] (v5) (v4) to [out=45, in=135] (v6);
        \end{tikzpicture}
\]
To compute $\ell_H(v)$, observe that there are $2$ edges, $\{2,3\}$ and $\{\bar3,\bar2\}$ (indicated by blue zigzags), whose right endpoints are in $\Gamma(I)_v$ and whose left endpoints are not. Since these two edges are in the same $\iota$-orbit, we have $\ell_H(v)=1$.
\end{example}

The chromatic quasisymmetric function $X_\Gamma$ is well-defined on any graph $\Gamma$ with totally ordered vertex set: an ascent of a coloring $\kappa$ is any edge $vw$ such that $v$ precedes $w$ in the given total order and $\kappa(v)<\kappa(w)$. For any $v\in {^J}W$, we consider the vertices of the induced subgraph $\Gamma(I)_v$ on $\{v^{-1}(j)\mid j\in [n]\}$ to have total order induced by that on $[n]$.   We now use the graph $\Gamma(I)$ to give a combinatorial formula for the restriction of the dot action representation on $H^*(\Hess(s,H(I)))$ to $W_J$. 

\begin{thm} \label{thm.graph}
Assume $\Phi$ is an irreducible root system of type $B_n$, $C_n$, or $D_n$, $I$ is an order ideal in $\Phi^+$, and let $H=H(I)$ be the associated Hessenberg space.  Let $s$ be a regular semisimple element of the associated Lie algebra.  Assume the vertex set of $\Gamma(I)$ is ordered by $\preceq$ as in \eqref{eq.order2}.  For $v \in {^J}W$, let 
\[
\ell_{\Gamma(I)}(v):= |\{ab \in E(I) : a \prec b,\, a \notin \{v^{-1}(j): j\in [n]\}, \, v(b) \in \{v^{-1}(j): j\in [n]\}\,  \}|
\]
and $\Gamma(I)_v$ be the induced subgraph of $\Gamma(I)$ on the vertex set $\{ v^{-1}(j) : j\in [n] \}$. 
Then
$$
\fch\left(\left(\sum_{j \geq 0} H^{2j}(\Hess(s,H))t^j\right)\Big\downarrow^W_{W_J}\right)=\sum_{v \in {^J}W}t^{\ell_{\Gamma(I)}(v)} \omega \left( X_{\Gamma(I)_v}(\vx;t)\right).
$$
\end{thm}

\begin{proof}
The theorem follows from Theorem~\ref{bcd}, Lemma~\ref{lem.graphdef}, and Corollary~\ref{cor.hlength}.
\end{proof}

Our next example uses Theorem~\ref{thm.graph} to compute the restriction of the dot action using $\Gamma(I)$.

\begin{example}\label{ex.B3.full}
Let $\Phi=B_3$, $I\subseteq \Phi^+$ be the order ideal generated by $e_1-e_3$ and $e_2$, and $H=H(I)$ be the Hessenberg space associated with $I$.
In this case, $J=\{e_1-e_2, e_2-e_3\}\subset \Delta$ and $W_J \cong S_3$.
We shall use Theorem~\ref{thm.graph} to compute the graded dot action representation on $H^*(\Hess(s,H(I)))$ restricted from $W(B_3)$ to $W_J$. The eight signed permutations $v\in {^J}W$ are listed in Example~\ref{ex:ParabolicQuotient}. For each $v\in {^J}W$, the subgraphs $\Gamma(I)_v$ and the corresponding chromatic quasisymmetric function $t^{\ell_{\Gamma(I)}(v)} X_{\Gamma(I)_v}(\vx;t)$ are displayed in Figure~\ref{fig: RestrictExample}. Computing in the basis of homogeneous symmetric functions, we have
\[
X_{\begin{tikzpicture}[
scale=0.3, vertex/.style={circle, fill, inner sep=1pt}
]
    \node[vertex] (a) at (0,0) {};
    \node[vertex] (b) at (1,0) {};
    \node[vertex] (c) at (2,0) {};

    \draw (a)--(b)--(c);
    \draw (a) [out=45, in=135] to (c);
\end{tikzpicture}}(\vx;t) = (1+2t+2t^2+t^3) h_3,\, \text{ and } \,  X_{\begin{tikzpicture}[
scale=0.3, vertex/.style={circle, fill, inner sep=1pt}
]
    \node[vertex] (a) at (0,0) {};
    \node[vertex] (b) at (1,0) {};
    \node[vertex] (c) at (2,0) {};

    \draw (a)--(b);
\end{tikzpicture}}(\vx;t) =  X_{\begin{tikzpicture}[
scale=0.3, vertex/.style={circle, fill, inner sep=1pt}
]
    \node[vertex] (a) at (0,0) {};
    \node[vertex] (b) at (1,0) {};
    \node[vertex] (c) at (2,0) {};

    \draw (b)--(c);
\end{tikzpicture}}(\vx;t) =  (1+t)h_{21}. \; 
\]
By Theorem \ref{thm.graph}, the desired graded Frobenius image is equal to  
\begin{align*}
    &\sum_{j\ge 0}\ch\left(H^{2j}(\Hess(s,H))\downarrow^{W(B_3)}_{S_3}\right)t^j\\
    =& \omega\left((1+t^2)X_{\begin{tikzpicture}[
scale=0.3, vertex/.style={circle, fill, inner sep=1pt}
]
    \node[vertex] (a) at (0,0) {};
    \node[vertex] (b) at (1,0) {};
    \node[vertex] (c) at (2,0) {};

    \draw (a)--(b)--(c);
    \draw (a) [out=45, in=135] to (c);
\end{tikzpicture}}(\vx;t)+(t+4t^2+t^3)X_{\begin{tikzpicture}[
scale=0.3, vertex/.style={circle, fill, inner sep=1pt}
]
    \node[vertex] (a) at (0,0) {};
    \node[vertex] (b) at (1,0) {};
    \node[vertex] (c) at (2,0) {};

    \draw (a)--(b);
\end{tikzpicture}}(\vx;t)\right)\\
    =& (1+t^2)(1+2t+2t^2+t^3)h_3+(t+4t^2+t^3)(1+t)h_{21}.
\end{align*}

\begin{figure}[htbp]
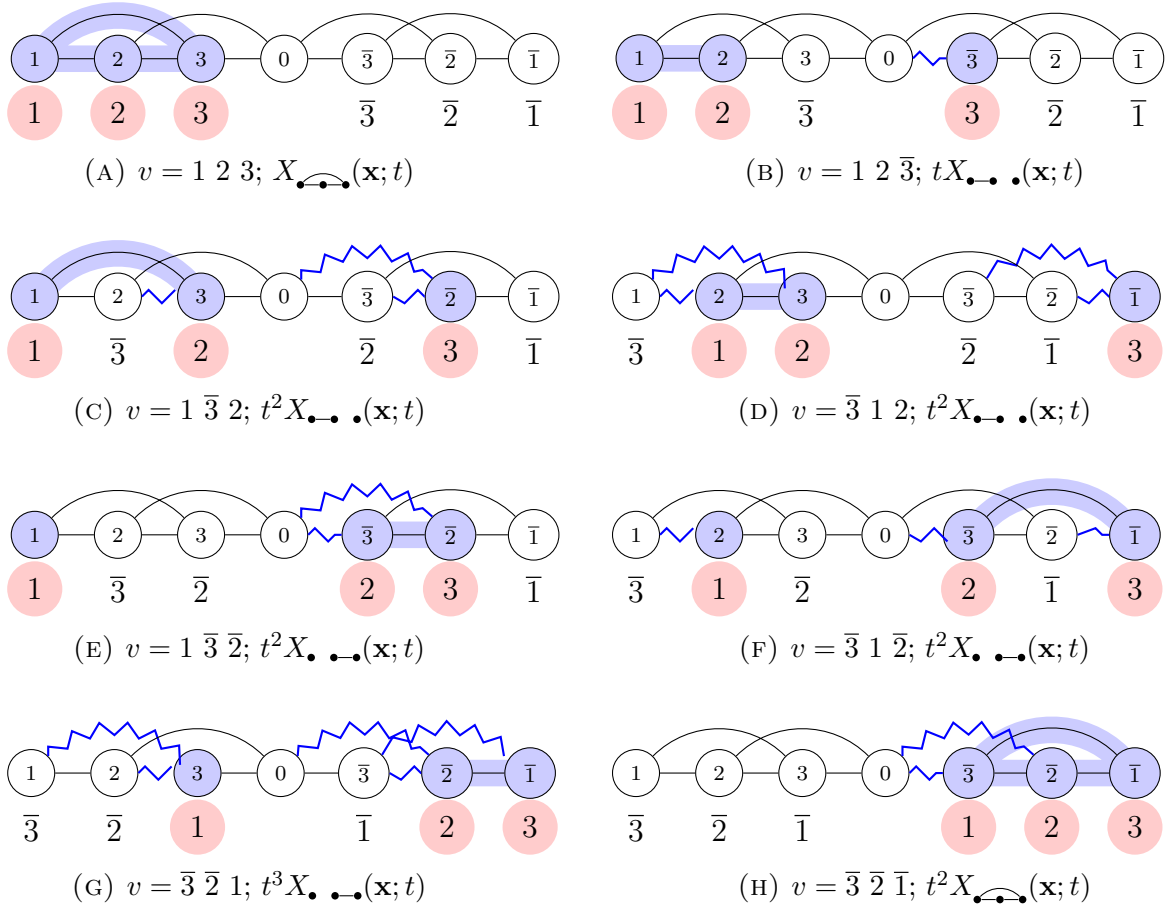

\centering
\begin{subfigure}{0.4 \textwidth}
    \begin{tikzpicture}[
        scale=0.55,
        vertex/.style={circle, draw, minimum size=0.55 cm},
        every edge/.style={draw, thick}
        ]
        \coordinate[circle, minimum size=0.55 cm] (p1) at (0,0);
        \coordinate[circle, minimum size=0.55 cm] (p2) at (2,0);
        \coordinate[circle, minimum size=0.55 cm] (p3) at (4,0);
        \coordinate[circle, minimum size=0.55 cm] (p0) at (6,0);
        \coordinate[circle, minimum size=0.55 cm] (p4) at (8,0);
        \coordinate[circle, minimum size=0.55 cm] (p5) at (10,0);
        \coordinate[circle, minimum size=0.55 cm] (p6) at (12,0);

        \path[draw=blue!20, line width=10,line cap=round, line join=round] 
        (p1) to [out=45, in=135] (p3) (p1) to (p2) to (p3);
            
        %Vertices
        \node[vertex, fill=blue!20, label=below:{}] (v1) at (0,0) {\tiny$1$};
        \node[vertex, fill=blue!20, label=below:{}] (v2) at (2,0) {\tiny$2$};
        \node[vertex, fill=blue!20, label=below:{}] (v3) at (4,0) {\tiny$3$};
        \node[vertex, label=below:{}] (v0) at (6,0) {\tiny$0$};
        \node[vertex, label=below:{}] (v4) at (8,0) {\tiny $\bar{3}$};
        \node[vertex, label=below:{}] (v5) at (10,0) {\tiny $\bar{2}$};
        \node[vertex, label=below:{}] (v6) at (12,0) {\tiny$\bar{1}$};

        %\node (w) at (-1.3,-1.3) {$v=$};
        \node[circle, fill=red!20, minimum size=0.3 cm] (w1) at (0,-1.3) {$1$};
        \node[circle, fill=red!20 , minimum size=0.3 cm] (w2) at (2,-1.3) {$2$};
        \node[circle, fill=red!20 , minimum size=0.3 cm] (w3) at (4,-1.3) {$3$};
        \node (w4) at (8,-1.3) {$\bar{3}$};
        \node (w5) at (10,-1.3) {$\bar{2}$};
        \node (w6) at (12,-1.3) {$\bar{1}$};

        \draw (v1)--(v2) (v3)--(v0) (v0)--(v4) (v5)--(v6);
        \draw (v2)--(v3) (v4)--(v5);
        \draw (v1) to [out=45, in=135] (v3) (v2) to [out=45, in=135] (v0);
        \draw (v0) to [out=45, in=135] (v5) (v4) to [out=45, in=135] (v6);

\end{tikzpicture}\noindent
    \caption{$v=1~2~3$; $X_{\input{Tikzpicture/A3}}(\vx;t)$}
\end{subfigure}
\hfill
\begin{subfigure}{0.5 \textwidth}
    \begin{tikzpicture}[
            scale=0.55,
            vertex/.style={circle, draw, minimum size=0.55 cm},
            every edge/.style={draw, thick}
            ]
        \coordinate[circle, minimum size=0.55 cm] (p1) at (0,0);
        \coordinate[circle, minimum size=0.55 cm] (p2) at (2,0);
        \coordinate[circle, minimum size=0.55 cm] (p3) at (4,0);
        \coordinate[circle, minimum size=0.55 cm] (p0) at (6,0);
        \coordinate[circle, minimum size=0.55 cm] (p4) at (8,0);
        \coordinate[circle, minimum size=0.55 cm] (p5) at (10,0);
        \coordinate[circle, minimum size=0.55 cm] (p6) at (12,0);

        \path[draw=blue!20, line width=10,line cap=round, line join=round] 
        (p1) to (p2);
            
        %Vertices
        \node[vertex, fill=blue!20, label=below:{}] (v1) at (0,0) {\tiny$1$};
        \node[vertex, fill=blue!20, label=below:{}] (v2) at (2,0) {\tiny$2$};
        \node[vertex, label=below:{}] (v3) at (4,0) {\tiny$3$};
        \node[vertex, label=below:{}] (v0) at (6,0) {\tiny$0$};
        \node[vertex, fill=blue!20, label=below:{}] (v4) at (8,0) {\tiny $\bar{3}$};
        \node[vertex, label=below:{}] (v5) at (10,0) {\tiny $\bar{2}$};
        \node[vertex, label=below:{}] (v6) at (12,0) {\tiny$\bar{1}$};

        %\node (w) at (-1.3,-1.3) {$v=$};
        \node[circle, fill=red!20, minimum size=0.3 cm] (w1) at (0,-1.3) {$1$};
        \node[circle, fill=red!20 , minimum size=0.3 cm] (w2) at (2,-1.3) {$2$};
        \node (w3) at (4,-1.3) {$\bar{3}$};
        \node[circle, fill=red!20 , minimum size=0.3 cm] (w4) at (8,-1.3) {$3$};
        \node (w5) at (10,-1.3) {$\bar{2}$};
        \node (w6) at (12,-1.3) {$\bar{1}$};

        \draw (v1)--(v2) (v3)--(v0) (v5)--(v6);
        \draw[thick, decorate, decoration=zigzag, blue!100] (v0)--(v4);
        \draw (v2)--(v3) (v4)--(v5);
        \draw (v1) to [out=45, in=135] (v3) (v2) to [out=45, in=135] (v0);
        \draw (v0) to [out=45, in=135] (v5) (v4) to [out=45, in=135] (v6);
    \end{tikzpicture}
    \caption{$v=1~2~\bar{3}$; $t X_{\input{Tikzpicture/A21}}(\vx;t)$}
\end{subfigure}\\

\

\begin{subfigure}{0.4\textwidth}
    \begin{tikzpicture}[
            scale=0.55,
            vertex/.style={circle, draw, minimum size=0.55 cm},
            every edge/.style={draw, thick}
            ]
        \coordinate[circle, minimum size=0.55 cm] (p1) at (0,0);
        \coordinate[circle, minimum size=0.55 cm] (p2) at (2,0);
        \coordinate[circle, minimum size=0.55 cm] (p3) at (4,0);
        \coordinate[circle, minimum size=0.55 cm] (p0) at (6,0);
        \coordinate[circle, minimum size=0.55 cm] (p4) at (8,0);
        \coordinate[circle, minimum size=0.55 cm] (p5) at (10,0);
        \coordinate[circle, minimum size=0.55 cm] (p6) at (12,0);

        \path[draw=blue!20, line width=10,line cap=round, line join=round] 
        (p1) to [out=45, in=135 ] (p3);
            
        %Vertices
        \node[vertex, fill=blue!20, label=below:{}] (v1) at (0,0) {\tiny$1$};
        \node[vertex, label=below:{}] (v2) at (2,0) {\tiny$2$};
        \node[vertex, fill=blue!20, label=below:{}] (v3) at (4,0) {\tiny$3$};
        \node[vertex, label=below:{}] (v0) at (6,0) {\tiny$0$};
        \node[vertex, label=below:{}] (v4) at (8,0) {\tiny $\bar{3}$};
        \node[vertex, fill=blue!20, label=below:{}] (v5) at (10,0) {\tiny $\bar{2}$};
        \node[vertex, label=below:{}] (v6) at (12,0) {\tiny$\bar{1}$};

        %\node (w) at (-1.3,-1.3) {$v=$};
        \node[circle, fill=red!20, minimum size=0.3 cm] (w1) at (0,-1.3) {$1$};
        \node (w2) at (2,-1.3) {$\bar{3}$};
        \node[circle, fill=red!20 , minimum size=0.3 cm] (w3) at (4,-1.3) {$2$};
        \node (w4) at (8,-1.3) {$\bar{2}$};
        \node[circle, fill=red!20 , minimum size=0.3 cm] (w5) at (10,-1.3) {$3$};
        \node (w6) at (12,-1.3) {$\bar{1}$};

        \draw (v1)--(v2) (v3)--(v0) (v0)--(v4) (v5)--(v6);
        \draw[thick, decorate, decoration=zigzag, blue!100]  (v2)--(v3) (v4)--(v5);
        \draw (v1) to [out=45, in=135] (v3) (v2) to [out=45, in=135] (v0);
        \draw[thick, decorate, decoration=zigzag, blue!100] (v0) to [out=45, in=135] (v5);
        \draw (v4) to [out=45, in=135] (v6);
    \end{tikzpicture}\noindent
    \caption{$v=1~\bar{3}~2$; $t^2 X_{\input{Tikzpicture/A21}}(\vx;t)$}  
\end{subfigure}
    \hfill
\begin{subfigure}{0.5 \textwidth}
    \begin{tikzpicture}[
            scale=0.55,
            vertex/.style={circle, draw, minimum size=0.55 cm},
            every edge/.style={draw, thick}
            ]
        \coordinate[circle, minimum size=0.55 cm] (p1) at (0,0);
        \coordinate[circle, minimum size=0.55 cm] (p2) at (2,0);
        \coordinate[circle, minimum size=0.55 cm] (p3) at (4,0);
        \coordinate[circle, minimum size=0.55 cm] (p0) at (6,0);
        \coordinate[circle, minimum size=0.55 cm] (p4) at (8,0);
        \coordinate[circle, minimum size=0.55 cm] (p5) at (10,0);
        \coordinate[circle, minimum size=0.55 cm] (p6) at (12,0);

        \path[draw=blue!20, line width=10,line cap=round, line join=round] 
        (p2) to  (p3);
            
        %Vertices
        \node[vertex, label=below:{}] (v1) at (0,0) {\tiny$1$};
        \node[vertex, fill=blue!20, label=below:{}] (v2) at (2,0) {\tiny$2$};
        \node[vertex, fill=blue!20, label=below:{}] (v3) at (4,0) {\tiny$3$};
        \node[vertex, label=below:{}] (v0) at (6,0) {\tiny$0$};
        \node[vertex, label=below:{}] (v4) at (8,0) {\tiny $\bar{3}$};
        \node[vertex, label=below:{}] (v5) at (10,0) {\tiny $\bar{2}$};
        \node[vertex, fill=blue!20, label=below:{}] (v6) at (12,0) {\tiny$\bar{1}$};

        %\node (w) at (-1.3,-1.3) {$v=$};
        \node (w1) at (0,-1.3) {$\bar{3}$};
        \node[circle, fill=red!20, minimum size=0.3 cm] (w2) at (2,-1.3) {$1$};
        \node[circle, fill=red!20 , minimum size=0.3 cm] (w3) at (4,-1.3) {$2$};
        \node (w4) at (8,-1.3) {$\bar{2}$};
        \node (w5) at (10,-1.3) {$\bar{1}$};
        \node[circle, fill=red!20 , minimum size=0.3 cm] (w6) at (12,-1.3) {$3$};

        \draw[thick, decorate, decoration=zigzag, blue!100] (v1)--(v2) (v5)--(v6);
        \draw  (v2)--(v3)--(v0)--(v4)--(v5);
        \draw[thick, decorate, decoration=zigzag, blue!100] (v1) to [out=45, in=135] (v3) (v4) to [out=45, in=135] (v6);
        \draw (v2) to [out=45, in=135] (v0);
        \draw (v0) to [out=45, in=135] (v5);
    \end{tikzpicture}\noindent
    \caption{$v=\bar{3}~1~2$; $t^2 X_{\input{Tikzpicture/A21}}(\vx;t)$}
\end{subfigure}\\

\

\begin{subfigure}{0.4 \textwidth}
    \begin{tikzpicture}[
        scale=0.55,
        vertex/.style={circle, draw, minimum size=0.55 cm},
        every edge/.style={draw, thick}
        ]
        \coordinate[circle, minimum size=0.55 cm] (p1) at (0,0);
        \coordinate[circle, minimum size=0.55 cm] (p2) at (2,0);
        \coordinate[circle, minimum size=0.55 cm] (p3) at (4,0);
        \coordinate[circle, minimum size=0.55 cm] (p0) at (6,0);
        \coordinate[circle, minimum size=0.55 cm] (p4) at (8,0);
        \coordinate[circle, minimum size=0.55 cm] (p5) at (10,0);
        \coordinate[circle, minimum size=0.55 cm] (p6) at (12,0);

        \path[draw=blue!20, line width=10,line cap=round, line join=round] 
        (p4) to (p5);
            
        %Vertices
        \node[vertex, fill=blue!20, label=below:{}] (v1) at (0,0) {\tiny$1$};
        \node[vertex, label=below:{}] (v2) at (2,0) {\tiny$2$};
        \node[vertex, label=below:{}] (v3) at (4,0) {\tiny$3$};
        \node[vertex, label=below:{}] (v0) at (6,0) {\tiny$0$};
        \node[vertex, fill=blue!20, label=below:{}] (v4) at (8,0) {\tiny $\bar{3}$};
        \node[vertex, fill=blue!20, label=below:{}] (v5) at (10,0) {\tiny $\bar{2}$};
        \node[vertex, label=below:{}] (v6) at (12,0) {\tiny$\bar{1}$};

        %\node (w) at (-1.3,-1.3) {$v=$};
        \node[circle, fill=red!20, minimum size=0.3 cm] (w1) at (0,-1.3) {$1$};
        \node (w2) at (2,-1.3) {$\bar{3}$};
        \node (w3) at (4,-1.3) {$\bar{2}$};
        \node[circle, fill=red!20 , minimum size=0.3 cm] (w4) at (8,-1.3) {$2$};
        \node[circle, fill=red!20 , minimum size=0.3 cm] (w5) at (10,-1.3) {$3$};
        \node (w6) at (12,-1.3) {$\bar{1}$};

        \draw (v1)--(v2) (v3)--(v0) (v5)--(v6);
        \draw[thick, decorate, decoration=zigzag, blue!100] (v0)--(v4);
        \draw (v2)--(v3) (v4)--(v5);
        \draw (v1) to [out=45, in=135] (v3) (v2) to [out=45, in=135] (v0);
        \draw (v4) to [out=45, in=135] (v6);
        \draw[thick, decorate, decoration=zigzag, blue!100]  (v0) to [out=45, in=135] (v5);
\end{tikzpicture}\noindent
    \caption{$v=1~\bar{3}~\bar{2}$; $t^2 X_{
    \input{Tikzpicture/A12}}(\vx;t)$}
\end{subfigure}
\hfill
\begin{subfigure}{0.5 \textwidth}
    \begin{tikzpicture}[
        scale=0.55,
        vertex/.style={circle, draw, minimum size=0.55 cm},
        every edge/.style={draw, thick}
        ]
        \coordinate[circle, minimum size=0.55 cm] (p1) at (0,0);
        \coordinate[circle, minimum size=0.55 cm] (p2) at (2,0);
        \coordinate[circle, minimum size=0.55 cm] (p3) at (4,0);
        \coordinate[circle, minimum size=0.55 cm] (p0) at (6,0);
        \coordinate[circle, minimum size=0.55 cm] (p4) at (8,0);
        \coordinate[circle, minimum size=0.55 cm] (p5) at (10,0);
        \coordinate[circle, minimum size=0.55 cm] (p6) at (12,0);

        \path[draw=blue!20, line width=10,line cap=round, line join=round] 
        (p4) to [out=45, in=135] (p6);
            
        %Vertices
        \node[vertex, label=below:{}] (v1) at (0,0) {\tiny$1$};
        \node[vertex, fill=blue!20, label=below:{}] (v2) at (2,0) {\tiny$2$};
        \node[vertex, label=below:{}] (v3) at (4,0) {\tiny$3$};
        \node[vertex, label=below:{}] (v0) at (6,0) {\tiny$0$};
        \node[vertex, fill=blue!20, label=below:{}] (v4) at (8,0) {\tiny $\bar{3}$};
        \node[vertex, label=below:{}] (v5) at (10,0) {\tiny $\bar{2}$};
        \node[vertex, fill=blue!20, label=below:{}] (v6) at (12,0) {\tiny$\bar{1}$};

        %\node (w) at (-1.3,-1.3) {$v=$};
        \node (w1) at (0,-1.3) {$\bar{3}$};
        \node[circle, fill=red!20, minimum size=0.3 cm] (w2) at (2,-1.3) {$1$};
        \node (w3) at (4,-1.3) {$\bar{2}$};
        \node (w4)[circle, fill=red!20 , minimum size=0.3 cm] at (8,-1.3) {$2$};
        \node (w5) at (10,-1.3) {$\bar{1}$};
        \node[circle, fill=red!20 , minimum size=0.3 cm] (w6) at (12,-1.3) {$3$};

        \draw (v3)--(v0);
        \draw[thick, decorate, decoration=zigzag, blue!100]  (v1)--(v2) (v0)--(v4) (v5)--(v6);
        \draw (v2)--(v3) (v4)--(v5);
        \draw (v1) to [out=45, in=135] (v3) (v2) to [out=45, in=135] (v0);
        \draw (v4) to [out=45, in=135] (v6)   (v0) to [out=45, in=135] (v5);

    \end{tikzpicture}
    \caption{$v=\bar{3}~1~\bar{2}$; $t^2 X_{\input{Tikzpicture/A12}}(\vx;t)$}
\end{subfigure}\\

\

\begin{subfigure}{0.4 \textwidth}
    \begin{tikzpicture}[
        scale=0.55,
        vertex/.style={circle, draw, minimum size=0.55 cm},
        every edge/.style={draw, thick}
        ]
        \coordinate[circle, minimum size=0.55 cm] (p1) at (0,0);
        \coordinate[circle, minimum size=0.55 cm] (p2) at (2,0);
        \coordinate[circle, minimum size=0.55 cm] (p3) at (4,0);
        \coordinate[circle, minimum size=0.55 cm] (p0) at (6,0);
        \coordinate[circle, minimum size=0.55 cm] (p4) at (8,0);
        \coordinate[circle, minimum size=0.55 cm] (p5) at (10,0);
        \coordinate[circle, minimum size=0.55 cm] (p6) at (12,0);

        \path[draw=blue!20, line width=10,line cap=round, line join=round] 
        (p5) to (p6);
            
        %Vertices
        \node[vertex, label=below:{}] (v1) at (0,0) {\tiny$1$};
        \node[vertex, label=below:{}] (v2) at (2,0) {\tiny$2$};
        \node[vertex, fill=blue!20, label=below:{}] (v3) at (4,0) {\tiny$3$};
        \node[vertex, label=below:{}] (v0) at (6,0) {\tiny$0$};
        \node[vertex, label=below:{}] (v4) at (8,0) {\tiny $\bar{3}$};
        \node[vertex, fill=blue!20, label=below:{}] (v5) at (10,0) {\tiny $\bar{2}$};
        \node[vertex, fill=blue!20, label=below:{}] (v6) at (12,0) {\tiny$\bar{1}$};

        %\node (w) at (-1.3,-1.3) {$v=$};
        \node (w1) at (0,-1.3) {$\bar{3}$};
        \node (w2) at (2,-1.3) {$\bar{2}$};
        \node[circle, fill=red!20, minimum size=0.3 cm] (w3) at (4,-1.3) {$1$};
        \node (w4) at (8,-1.3) {$\bar{1}$};
        \node (w5)[circle, fill=red!20 , minimum size=0.3 cm] at (10,-1.3) {$2$};
        \node[circle, fill=red!20 , minimum size=0.3 cm] (w6) at (12,-1.3) {$3$};

        \draw (v1)--(v2) (v3)--(v0)--(v4) (v5)--(v6);
        \draw[thick, decorate, decoration=zigzag, blue!100]  (v2)--(v3) (v4)--(v5);
        \draw[thick, decorate, decoration=zigzag, blue!100] (v1) to [out=45, in=135] (v3) (v4) to [out=45, in=135] (v6)  (v0) to [out=45, in=135] (v5);
        \draw (v2) to [out=45, in=135] (v0);

    \end{tikzpicture}\noindent
    \caption{$v=\bar{3}~\bar{2}~1$; $t^3 X_{\input{Tikzpicture/A12}}(\vx;t)$}
\end{subfigure}
\hfill
\begin{subfigure}{0.5 \textwidth}
    \begin{tikzpicture}[
        scale=0.55,
        vertex/.style={circle, draw, minimum size=0.55 cm},
        every edge/.style={draw, thick}
        ]
        \coordinate[circle, minimum size=0.55 cm] (p1) at (0,0);
        \coordinate[circle, minimum size=0.55 cm] (p2) at (2,0);
        \coordinate[circle, minimum size=0.55 cm] (p3) at (4,0);
        \coordinate[circle, minimum size=0.55 cm] (p0) at (6,0);
        \coordinate[circle, minimum size=0.55 cm] (p4) at (8,0);
        \coordinate[circle, minimum size=0.55 cm] (p5) at (10,0);
        \coordinate[circle, minimum size=0.55 cm] (p6) at (12,0);

        \path[draw=blue!20, line width=10,line cap=round, line join=round] 
        (p4) to (p5) to (p6) (p4) to [out=45, in=135] (p6);
            
        %Vertices
        \node[vertex, label=below:{}] (v1) at (0,0) {\tiny$1$};
        \node[vertex, label=below:{}] (v2) at (2,0) {\tiny$2$};
        \node[vertex, label=below:{}] (v3) at (4,0) {\tiny$3$};
        \node[vertex, label=below:{}] (v0) at (6,0) {\tiny$0$};
        \node[vertex, fill=blue!20, label=below:{}] (v4) at (8,0) {\tiny $\bar{3}$};
        \node[vertex, fill=blue!20, label=below:{}] (v5) at (10,0) {\tiny $\bar{2}$};
        \node[vertex, fill=blue!20, label=below:{}] (v6) at (12,0) {\tiny$\bar{1}$};

        %\node (w) at (-1.3,-1.3) {$v=$};
        \node (w1) at (0,-1.3) {$\bar{3}$};
        \node (w2) at (2,-1.3) {$\bar{2}$};
        \node (w3) at (4,-1.3) {$\bar{1}$};
        \node[circle, fill=red!20 , minimum size=0.3 cm] (w4) at (8,-1.3) {$1$};
        \node[circle, fill=red!20 , minimum size=0.3 cm] (w5) at (10,-1.3) {$2$};
        \node[circle, fill=red!20, minimum size=0.3 cm] (w6) at (12,-1.3) {$3$};

        \draw (v1)--(v2) (v3)--(v0) (v5)--(v6);
        \draw[thick, decorate, decoration=zigzag, blue!100] (v0)--(v4);
        \draw (v2)--(v3) (v4)--(v5);
        \draw (v1) to [out=45, in=135] (v3) (v2) to [out=45, in=135] (v0);
        \draw (v4) to [out=45, in=135] (v6);
        \draw[thick, decorate, decoration=zigzag, blue!100]  (v0) to [out=45, in=135] (v5);
\end{tikzpicture}
    \caption{$v=\bar{3}~\bar{2}~\bar{1}$; $t^2 X_{\input{Tikzpicture/A3}}(\vx;t)$}
\end{subfigure}
\caption{Subgraphs $\Gamma(I)_v$ and summands of the graded character computed in Example~\ref{ex.B3.full}.}
\label{fig: RestrictExample}
\end{figure}
\end{example}

\begin{rem}\label{rem.P-tab} For any ideal $I$ in $\Phi^+(A_{n-1})$, there is a combinatorial formula for the graded Schur expansion of $X_{\Gamma(I)}(\mathbf{x};t)$ in terms of ``$\cp(I)$-tableaux''; see~\cite[\S 6]{Shareshian-Wachs2016}. Thus, the results of Theorem~\ref{thm.graph} yield an interpretation of the character of the restriction in terms of the various $\cp(v(I)\cap \Phi_J^+)$-tableaux with entries from the set $\{v^{-1}(j) : j\in [n]\}$ for $v\in {^J}W$. 
\end{rem}

%%%%%%%%%%%%%%%%%%%%%%%

\section{A conjecture of Lesnevich}\label{sec.Lesnevich}  In this section we use Theorem~\ref{thm.graph} to prove a conjecture of Lesnevich~\cite{Lesnevich2025} about the dot action representation in types $B_n$ and $C_n$.  To begin, we review some well-known and relevant facts.

If $\Phi=B_n$, we denote the set of simple roots by $\Delta=\{\beta_1,\ldots,\beta_n\}$, with $\beta_i = e_i-e_{i+1}$ when $i<n$ and $\beta_n = e_n$.

The set $\Phi^+(B_n)$ of positive roots consists of
\begin{align*}
e_i-e_j&=\sum_{k=i}^{j-1}\beta_k  & 1 \leq i<j \leq n, \\ 
e_i&=\sum_{k=i}^n \beta_k  & 1 \leq i \leq n,\\
e_i+e_j&=\sum_{k=i}^n\beta_k+\sum_{k=0}^{n-j}\beta_{n-k} & 1\leq i<j\leq n.
\end{align*}

If $\Phi=C_n$, we denote the set of simple roots by $\Delta=\{\gamma_1,\ldots,\gamma_n\}$, with $\gamma_i = e_i-e_{i+1}$ when $i<n$ and $\gamma_n = 2e_n$.
The set $\Phi^+(C_n)$ of positive roots consists of
\begin{align*}
e_i-e_j&=\sum_{k=i}^{j-1}\gamma_k & 1 \leq i<j \leq n, \\ 
e_i+e_n&=\sum_{k=i}^n \gamma_k & 1 \leq i \leq n-1, \\ 
e_i+e_{j-1}&=\sum_{k=i}^n\gamma_k+\sum_{k=0}^{n-j}\gamma_{n-k-1} & 1\leq i<j \leq n.
\end{align*}

Recall that, in either case, $\Phi^+$ is partially ordered, according to the rule that $\theta$ covers $\rho$ in $\Phi^+$ if $\theta-\rho\in \Delta$. For example, the Hasse diagrams of $\Phi^+(C_3)$ and $\Phi^+(B_3)$ appear in Figures~\ref{fig:C3} and~\ref{fig:B3}, respectively. Although $B_n$ and $C_n$ are not isomorphic root systems, the posets $\Phi^+(B_n)$ and $\Phi^+(C_n)$ are isomorphic for all $n$. Let $\pi:\Phi^+(B_n) \rightarrow \Phi^+(C_n)$ be the natural poset isomorphism, which is defined by 
\begin{align*}
   \pi(e_i-e_j)&=e_i-e_j  & 1\le i<j\le n,\\
   \pi(e_i)&=e_i+e_n & 1\le i\le n,\\
   \pi(e_i+e_j)&= e_i+e_{j-1} & 1\le i<j\le n.
\end{align*}

Let $s\in \mathfrak{so}_{2n+1}(\C)$ and $s'\in \mathfrak{sp}_{2n}(\C)$ be regular semisimple elements.
Since the Weyl groups $W(B_n)$ and $W(C_n)$ are equal, one can compare the dot action representations on the cohomology of the corresponding regular semisimple Hessenberg varieties $\cb(s,H(I))$ and $\cb(s^\prime,H(\pi(I)))$ for any ideal $I \subseteq \Phi^+(B_n)$.

For each root $\alpha \in \Phi^+$, write $s_\alpha$ for the associated reflection.  Given an order ideal $I \subseteq \Phi^+$, define $R(I)$ to be the set of reflections associated with $I$,
$$
R(I):=\{s_\alpha:\alpha \in I\}.
$$
In \cite[Lemma 3.1]{Lesnevich2025}, Lesnevich shows that, given an ideal $I \subseteq \Phi^+(B_n)$, if $R(I)=R(\pi(I))$ then the equivariant cohomology rings of $\Hess(s, H(I))$ and $\Hess(s', H(\pi(I)))$ are isomorphic. Consequently, the corresponding dot action representations on are also isomorphic in this case.  Lesnevich also suggests that the converse of his result should hold~\cite[Remark 3.2]{Lesnevich2025}, and we prove here that this suggestion is correct.

\begin{thm} \label{thm.nate} Let $I \subseteq \Phi^+(B_n)$ be an order ideal and $\pi: \Phi^+(B_n) \to \Phi^+(C_n)$ the poset isomorphism defined above.  Assume that $s\in \mathfrak{so}_{2n+1}(\C)$ and $s^\prime \in \mathfrak{sp}_{2n}(\C)$ are regular semisimple elements. The dot action representations $H^\ast(\Hess(s, H(I)))$ and $H^\ast(\Hess(s^\prime,H(\pi(I))))$ are isomorphic if and only if $R(I) = R(\pi(I))$. 
\end{thm}

The backward direction of Theorem~\ref{thm.nate} is Lesnevich's~\cite[Lemma 3.1]{Lesnevich2025}. The forward direction is an immediate consequence of the following.

\begin{prop} \label{prop.nate}
Let $I \subseteq \Phi^+(B_n)$ be an order ideal.  Assume that $s\in \mathfrak{so}_{2n+1}(\C)$ and $s^\prime \in \mathfrak{sp}_{2n}(\C)$ are regular semisimple elements.  If $R(I) \neq R(\pi(I))$ then the multiplicity of the trivial representation of $W_J \simeq S_n$ in the restriction of the (ungraded) type $B_n$ dot action representation
\[
H^\ast(\cb(s,H(I)))\big\downarrow^{W(B_n)}_{W_J}
\]
is strictly less than the multiplicity of the trivial representation of $W_J \simeq S_n$ in the restriction of the (ungraded) type $C_n$ dot action representation 
\[
H^\ast(\cb(s^\prime,H(\pi(I))))\big\downarrow^{W(C_n)}_{W_J}.
\]
\end{prop}

We prove Proposition~\ref{prop.nate} using two lemmas.  Let us fix an ideal $I \subseteq \Phi^+(B_n)$.  We have graphs $\Gamma(I)$ and $\Gamma(\pi(I))$ on vertex set $[n] \cup \{0\} \cup [\bar{n}]$.  We define graphs $\Gamma_0(I)$ and $\Gamma_0(\pi(I))$ on vertex set $[n] \cup [\bar{n}]$ as
\begin{itemize}
    \item $\Gamma_0(I)$ is the subgraph of $\Gamma(I)$ induced on $[n] \cup [\bar{n}]$; and
    \item $\Gamma_0(\pi(I))$ is obtained from the subgraph of $\Gamma(\pi(I))$ induced on $[n] \cup [\bar{n}]$ by removing all edges of the form $\{i,\bar{i}\}$.
\end{itemize}
Let $v\in {^J}W$. We observe that, since $v$ is a signed permutation in $W=W(B_n)=W(C_n)$, $\Gamma(I)_v$ is an induced subgraph of $\Gamma_0(I)$ and $\Gamma(\pi(I))_v$ is an induced subgraph of $\Gamma_0(\pi(I))$.

\begin{example} Consider the order ideal $I\subseteq \Phi^+(B_3)$ generated by $e_1-e_3$ and $2e_2$ and the corresponding graph $\Gamma(I)$, as pictured in Figure~\ref{fig:B3}. The ideal $\pi(I)$ and graph $\Gamma(\pi(I))$ are pictured in Figure~\ref{fig:C3}. The graphs $\Gamma_0(I)$ and $\Gamma_0(\pi(I))$ are both equal to the following. 
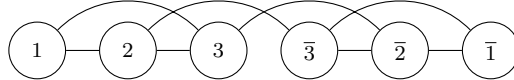
\begin{center}
    \begin{tikzpicture}[
                scale=0.6,
                vertex/.style={circle, draw, minimum size=0.75 cm},
                every edge/.style={draw, thick}
            ]
            %Vertices
            \node[vertex, label=below:{}] (v1) at (0,0) {\tiny $1$};
            \node[vertex, label=below:{}] (v2) at (2,0) {\tiny$2$};
            \node[vertex, label=below:{}] (v3) at (4,0) {\tiny$3$};
            \node[vertex, label=below:{}] (v4) at (6,0) {\tiny$\bar 3$};
            \node[vertex, label=below:{}] (v5) at (8,0) {\tiny$\bar 2$};
            \node[vertex, label=below:{}] (v6) at (10,0) {\tiny$\bar 1$};
            
            \draw (v1)--(v2)--(v3);
            \draw (v4)--(v5)--(v6);
            \draw (v1) to [out=45, in=135] (v3) (v2) to [out=45, in=135] (v4);
            \draw (v3) to [out=45, in=135] (v5) (v4) to [out=45, in=135] (v6);
        \end{tikzpicture}  
\end{center}
Notice that $R(I)=R(\pi(I))$ in this case.
\end{example}

\begin{lemma} \label{lem.subgraph}
For every ideal $I \subseteq \Phi^+(B_n)$, both
\begin{enumerate}
    \item $\Gamma_0(I)$ is a subgraph of $\Gamma_0(\pi(I))$, and
    \item $\Gamma_0(I)=\Gamma_0(\pi(I))$ if and only if $R(I)=R(\pi(I))$
\end{enumerate}
hold.
\end{lemma}

\begin{proof}
For any lower order ideal $L$ in either $\Phi^+(B_n)$ or $\Phi^+(C_n)$, we define
$$
L^-:=L \cap \{e_i-e_j:1 \leq i<j \leq n\},
$$
$$
L^+:=L \cap \{e_i+e_j:1 \leq i <j \leq n\},
$$
and
$$
L^0:=\left\{ \begin{array}{ll} L \cap \{e_i:1 \leq i \leq n\} & L \subseteq \Phi^+(B_n) \\ L \cap \{2e_i:1 \leq i \leq n\} & L \subseteq \Phi^+(C_n). \end{array} \right.
$$
We have that $L = L^-\sqcup L^+ \sqcup L^0$ is a partition of $L$.
By definition, the graph $\Gamma_0(L)$ is obtained by $\Gamma(L)$ by deleting edges corresponding to the subset~$L^0$ in $L$.

Now we fix an ideal $I \subseteq \Phi^+(B_n)$.  It follows directly from the definitions that
\begin{itemize}
\item[(a)] $\pi(I)^-=I^-$.
\end{itemize}
We claim that
\begin{itemize}
\item[(b)] $I^+ \subseteq \pi(I)^+$, and
\item[(c)] $I^+=\pi(I)^+$ if and only if $R(I)=R(\pi(I))$.
\end{itemize}
Together, (a),(b), and (c) imply the lemma.  So, it remains to prove (b) and (c).

We begin with (b). 
If $e_i+e_j\in I$, then $\pi(e_i+e_j)=e_i+e_{j-1}\in\pi(I)$. Since $\pi(I)$ is a lower order ideal of $\Phi^+(C_n)$ and
\[
e_i+e_{j} = (e_i+e_{j-1}) - \gamma_{j-1}  \text{ for }\; \gamma_{j-1}\in \Delta
\]
we get $e_i+e_{j-1}>e_i+e_j$ in $\Phi^+(C_n)$, and thus $e_i+e_j\in\pi(I)$ as well.

Now suppose $R(I)=R(\pi(I))$. 
If $\alpha= e_i+e_j\in \pi(I)^+$ then $s_\alpha\in R(\pi(I))$ and thus $s_\alpha\in R(I)$ and $\alpha = e_i+e_j \in I^+$. Thus, it is straightforward to see that $I^+=\pi(I)^+$ if $R(I)=R(\pi(I))$.  Statement (c) now follows once we show that $R(I)=R(\pi(I))$ whenever $I^+=\pi(I)^+$.  We observe that if $I^+=\pi(I)^+$ then $R(I)=R(\pi(I))$ if and only if $\pi(I)^0=\{2e_i: e_i \in I^0\}$.  We also know that $|\pi(I)|=|I|$ and $|\pi(I)^-|=|I^-|$ by (a). Our assumption that $I^+=\pi(I)^+$ further implies $|I^0|=|\pi(I)^0|$. Thus, the sets $\pi(I)^0$ and $\{2e_i: e_i \in I^0\}$ have the same cardinality. The fact that they are equal follows immediately from the facts that 
\begin{itemize}
    \item both $I$ and $\pi(I)$ are ideals,
    \item $e_i>e_{i+1}$ in $\Phi^+(B_n)$ for all $i \in [n-1]$, and
    \item $2e_i>2e_{i+1}$ in $\Phi^+(C_n)$ for all $i \in [n-1]$.
\end{itemize}
The last two claims are straightforward to verify. This completes the proof of (c).
\end{proof}

\begin{lemma} \label{lem.triv}
Let $I$ be an ideal in $\Phi^+(A_{n-1}) = \{e_i-e_j \mid 1\leq i<j\leq n\}$, $\Gamma(I)$ be the corresponding naturally labeled unit interval graph, and for $i \in [n-1]$ set 
$$
a_i(I):=|\{j>i:e_i-e_j \in I\}| = |\{j>i : ij\in E(I)\}|.
$$
Then the multiplicity of trivial representations in the (ungraded) dot action representation corresponding to $I$ is given by
$$
\langle \omega X_{\Gamma(I)}(\vx;1),s_n \rangle=\prod_{j=1}^{n-1}(1+a_j(I)).
$$
\end{lemma}

\begin{proof}
The lemma follows immediately upon setting $t=1$ in~\cite[Theorem 6.9]{Shareshian-Wachs2016}.     
\end{proof}

It is now straightforward to prove Proposition~\ref{prop.nate}.
\begin{proof}[Proof of Proposition~\ref{prop.nate}] Lemmas~\ref{lem.subgraph} and~\ref{lem.triv} together imply that 
$$
\langle \omega X_{\Gamma(I)_u},s_n \rangle \leq \langle \omega X_{\Gamma(\pi(I))_u},s_n \rangle \; \text{ for all } \;  u\in {^J}W. 
$$
If $R(I) \neq R(\pi(I))$ then there is some edge $ij$ in $\Gamma_0(\pi(I))$ but not in $\Gamma_0(I)$, by Lemma \ref{lem.subgraph}.  Such an edge appears in $\Gamma(\pi(I))_v$ but not in $\Gamma(I)_v$ for some $v \in {^J}W$.  Applying Lemmas \ref{lem.subgraph} and \ref{lem.triv}, we see that for such $v$ we have
$$
\langle \omega X_{\Gamma(I)_v},s_n \rangle < \langle \omega X_{\Gamma(\pi(I))_v},s_n \rangle.
$$
Proposition~\ref{prop.nate} now follows directly from Theorem \ref{thm.graph} (with $t=1$).
\end{proof}

\begin{example}\label{ex.B3C3.compare}
Let $I\subseteq\Phi^+(B_3)$ be the order ideal generated by $e_1-e_3$ and $e_2$. 
Then $\pi(I)\subseteq\Phi^+(C_3)$ is the ideal generated by $e_1-e_3$ and $e_2+e_3$. The relevant ideals are highlighted in the Hasse diagrams below, with $\Phi^+(B_3)$ pictured on the left and $\Phi^+(C_3)$ on the right. 
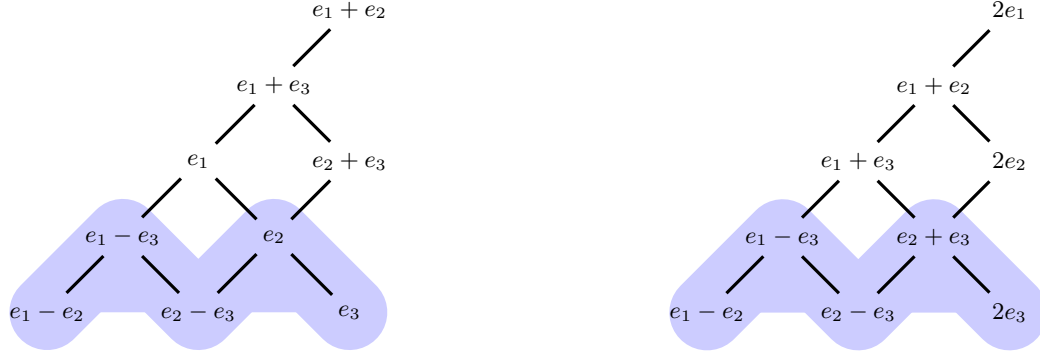
\begin{center}
    \vspace*{0.5em}
    % First Diagram
    \begin{minipage}{0.45\textwidth}
        \centering
        \begin{tikzpicture}[vertex/.style={font=\scriptsize}]
            \def\df{9}
            \def\w{1 cm}

            \coordinate (100) at (0,0);
            \coordinate (010) at (2,0);
            \coordinate (001) at (4,0);
            \coordinate (110) at (1,1);
            \coordinate (011) at (3,1);
            \coordinate (111) at (2,2);
            \coordinate (021) at (4,2);
            \coordinate (121) at (3,3);
            \coordinate (221) at (4,4);
            
            \path[fill=blue!20, draw=blue!20, line width=\w,line cap=round, line join=round] 
            (100) to (110) to (010) to (011) to (001);

            \node[vertex] (p100) at (100) {$e_1-e_2$};
            \node[vertex] (p010) at (010) {$e_2-e_3$};
            \node[vertex] (p001) at (001) {$e_3$};
            \node[vertex] (p110) at (110) {$e_1-e_3$};
            \node[vertex] (p011) at (011) {$e_2$};
            \node[vertex] (p111) at (111) {$e_1$};
            \node[vertex] (p121) at (121) {$e_1+e_3$};
            \node[vertex] (p021) at (021) {$e_2+e_3$};
            \node[vertex] (p221) at (221) {$e_1+e_2$};
        
            \draw[very thick] (p100)--(p110)--(p111)--(p121)--(p221);
            \draw[very thick] (p110)--(p010)--(p011)--(p021)--(p121);
            \draw[very thick] (p001)--(p011)--(p111);
        \end{tikzpicture}
    \end{minipage}%
    \hfill
    % Second Diagram
    \begin{minipage}{0.45\textwidth}
        \centering
        \begin{tikzpicture}[vertex/.style={font=\scriptsize}]
            \def\df{9}
            \def\w{1 cm}

            \coordinate (100) at (0,0);
            \coordinate (010) at (2,0);
            \coordinate (001) at (4,0);
            \coordinate (110) at (1,1);
            \coordinate (011) at (3,1);
            \coordinate (111) at (2,2);
            \coordinate (021) at (4,2);
            \coordinate (121) at (3,3);
            \coordinate (221) at (4,4);
            
            \path[fill=blue!20, draw=blue!20, line width=\w,line cap=round, line join=round] 
            (100) to (110) to (010) to (011) to (001);

            \node[vertex] (p100) at (100) {$e_1-e_2$};
            \node[vertex] (p010) at (010) {$e_2-e_3$};
            \node[vertex] (p001) at (001) {$2e_3$};
            \node[vertex] (p110) at (110) {$e_1-e_3$};
            \node[vertex] (p011) at (011) {$e_2+e_3$};
            \node[vertex] (p111) at (111) {$e_1+e_3$};
            \node[vertex] (p121) at (121) {$e_1+e_2$};
            \node[vertex] (p021) at (021) {$2e_2$};
            \node[vertex] (p221) at (221) {$2e_1$};
        
            \draw[very thick] (p100)--(p110)--(p111)--(p121)--(p221);
            \draw[very thick] (p110)--(p010)--(p011)--(p021)--(p121);
            \draw[very thick] (p001)--(p011)--(p111);
        \end{tikzpicture}
    \end{minipage}
    \vspace*{0.5em}
\end{center}
We see easily that $R(I)\neq R(\pi(I))$ since $e_2\in I$ and $2e_2\notin \pi(I)$, but these roots determine the same reflection. The graphs $\Gamma_0(I)$ and $\Gamma_0(\pi(I))$ are pictured below; we have $\Gamma_0(I)\subsetneq \Gamma_0(\pi(I))$ in this case.
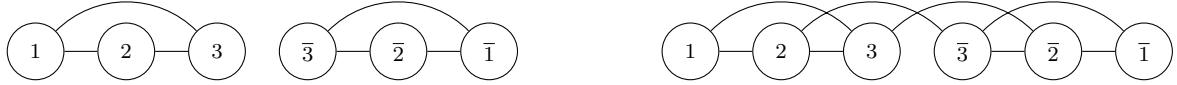
\begin{center}
    \vspace*{0.5em}
    % First Graph
    \begin{minipage}{0.45\textwidth}
        \centering
        \begin{tikzpicture}[
                scale=0.6,
                vertex/.style={circle, draw, minimum size=0.75 cm},
                every edge/.style={draw, thick}
            ]
            %Vertices
            \node[vertex, label=below:{}] (v1) at (0,0) {\tiny $1$};
            \node[vertex, label=below:{}] (v2) at (2,0) {\tiny$2$};
            \node[vertex, label=below:{}] (v3) at (4,0) {\tiny$3$};
            \node[vertex, label=below:{}] (v4) at (6,0) {\tiny$\bar 3$};
            \node[vertex, label=below:{}] (v5) at (8,0) {\tiny$\bar 2$};
            \node[vertex, label=below:{}] (v6) at (10,0) {\tiny$\bar1$};
            
            \draw (v1)--(v2)--(v3); 
            \draw (v4)--(v5)--(v6);
            \draw (v1) to [out=45, in=135] (v3) (v4) to [out=45, in=135] (v6);
        \end{tikzpicture} 
    \end{minipage}%
    \hfill
    % Second Graph
    \begin{minipage}{0.45\textwidth}
        \centering
        \begin{tikzpicture}[
                scale=0.6,
                vertex/.style={circle, draw, minimum size=0.75 cm},
                every edge/.style={draw, thick}
            ]
            %Vertices
            \node[vertex, label=below:{}] (v1) at (0,0) {\tiny $1$};
            \node[vertex, label=below:{}] (v2) at (2,0) {\tiny$2$};
            \node[vertex, label=below:{}] (v3) at (4,0) {\tiny$3$};
            \node[vertex, label=below:{}] (v4) at (6,0) {\tiny$\bar 3$};
            \node[vertex, label=below:{}] (v5) at (8,0) {\tiny$\bar 2$};
            \node[vertex, label=below:{}] (v6) at (10,0) {\tiny$\bar 1$};
            
            \draw (v1)--(v2)--(v3);
            \draw (v4)--(v5)--(v6);
            \draw (v1) to [out=45, in=135] (v3) (v2) to [out=45, in=135] (v4);
            \draw (v3) to [out=45, in=135] (v5) (v4) to [out=45, in=135] (v6);
        \end{tikzpicture}
    \end{minipage}
    \vspace*{0.5em}
\end{center}
For any $v\in{}^JW$, we have $\langle \omega X_{\Gamma(I)_v},s_3 \rangle \le \langle \omega X_{\Gamma(\pi(I))_v},s_3 \rangle$, and there exists some $v\in{}^JW$ that makes the inequality strict.

For example, when $w=e=1~2~3~\bar{3}~\bar{2}~\bar{1}\in{}^JW$, compare the two graphs Figure~\ref{fig:B=C}. 
In this case, the subgraphs $\Gamma(I)_{e}$ and $\Gamma(\pi(I))_{e}$ are equal and both are the naturally labeled unit interval graph corresponding to the order ideal $L=\Phi_J^+=\Phi^+(A_2)$.  We observe that $a_1(L)=2$, $a_2(L)=1$, and $a_3(L)=0$, so by Lemma~\ref{lem.triv} we have
\[
    \langle \omega X_{\Gamma(I)_{123}},s_3 \rangle=\langle \omega X_{\Gamma(\pi(I))_{123}},s_3 \rangle=(2+1)(1+1)(0+1)=6.
\]
\begin{figure}[ht]
    \centering
    \begin{subfigure}{0.4\textwidth}
    \begin{tikzpicture}[
        scale=0.55,
        vertex/.style={circle, draw, minimum size=0.45 cm},
        every edge/.style={draw, thick}
        ]
        \coordinate[circle, minimum size=0.45 cm] (p1) at (0,0);
        \coordinate[circle, minimum size=0.45 cm] (p2) at (2,0);
        \coordinate[circle, minimum size=0.45 cm] (p3) at (4,0);
        \coordinate[circle, minimum size=0.45 cm] (p0) at (6,0);
        \coordinate[circle, minimum size=0.45 cm] (p4) at (8,0);
        \coordinate[circle, minimum size=0.45 cm] (p5) at (10,0);
        \coordinate[circle, minimum size=0.45 cm] (p6) at (12,0);

        \path[draw=blue!20, line width=10,line cap=round, line join=round] 
        (p1) to [out=45, in=135] (p3) (p1) to (p2) to (p3);
            
        %Vertices
        \node[vertex, fill=blue!20, label=below:{}] (v1) at (0,0) {\tiny$1$};
        \node[vertex, fill=blue!20, label=below:{}] (v2) at (2,0) {\tiny$2$};
        \node[vertex, fill=blue!20, label=below:{}] (v3) at (4,0) {\tiny$3$};
        \node[vertex, label=below:{}] (v0) at (6,0) {\tiny$0$};
        \node[vertex, label=below:{}] (v4) at (8,0) {\tiny $\bar{3}$};
        \node[vertex, label=below:{}] (v5) at (10,0) {\tiny $\bar{2}$};
        \node[vertex, label=below:{}] (v6) at (12,0) {\tiny$\bar{1}$};

        %\node (w) at (-1.3,-1.3) {$v=$};
        \node[circle, fill=red!20, minimum size=0.25 cm] (w1) at (0,-1.3) {\tiny$1$};
        \node[circle, fill=red!20 , minimum size=0.25 cm] (w2) at (2,-1.3) {\tiny$2$};
        \node[circle, fill=red!20 , minimum size=0.25 cm] (w3) at (4,-1.3) {\tiny$3$};
        \node (w4) at (8,-1.3) {\tiny$\bar{3}$};
        \node (w5) at (10,-1.3) {\tiny$\bar{2}$};
        \node (w6) at (12,-1.3) {\tiny$\bar{1}$};

        \draw (v1)--(v2) (v3)--(v0) (v0)--(v4) (v5)--(v6);
        \draw (v2)--(v3) (v4)--(v5);
        \draw (v1) to [out=45, in=135] (v3) (v2) to [out=45, in=135] (v0);
        \draw (v0) to [out=45, in=135] (v5) (v4) to [out=45, in=135] (v6);

\end{tikzpicture}\noindent
    \caption{$\Gamma(I)_{e}\subset \Gamma(I)$}  
\end{subfigure}
    \hfill
\begin{subfigure}{0.4 \textwidth}
    \begin{tikzpicture}[
        scale=0.55,
        vertex/.style={circle, draw, minimum size=0.45 cm},
        every edge/.style={draw, thick}
        ]
        \coordinate[circle, minimum size=0.45 cm] (p1) at (0,0);
        \coordinate[circle, minimum size=0.45 cm] (p2) at (2,0);
        \coordinate[circle, minimum size=0.45 cm] (p3) at (4,0);
        \coordinate[circle, minimum size=0.45 cm] (p4) at (6,0);
        \coordinate[circle, minimum size=0.45 cm] (p5) at (8,0);
        \coordinate[circle, minimum size=0.45 cm] (p6) at (10,0);

        \path[draw=blue!20, line width=10,line cap=round, line join=round] 
        (p1) to [out=45, in=135] (p3) (p1) to (p2) to (p3);
            
        %Vertices
        \node[vertex, fill=blue!20, label=below:{}] (v1) at (0,0) {\tiny$1$};
        \node[vertex, fill=blue!20, label=below:{}] (v2) at (2,0) {\tiny$2$};
        \node[vertex, fill=blue!20, label=below:{}] (v3) at (4,0) {\tiny$3$};
        \node[vertex, label=below:{}] (v4) at (6,0) {\tiny $\bar{3}$};
        \node[vertex, label=below:{}] (v5) at (8,0) {\tiny $\bar{2}$};
        \node[vertex, label=below:{}] (v6) at (10,0) {\tiny$\bar{1}$};

        %\node (w) at (-1.3,-1.3) {$v=$};
        \node[circle, fill=red!20, minimum size=0.25 cm] (w1) at (0,-1.3) {\tiny$1$};
        \node[circle, fill=red!20 , minimum size=0.25 cm] (w2) at (2,-1.3) {\tiny$2$};
        \node[circle, fill=red!20 , minimum size=0.25 cm] (w3) at (4,-1.3) {\tiny$3$};
        \node (w4) at (6,-1.3) {\tiny$\bar{3}$};
        \node (w5) at (8,-1.3) {\tiny$\bar{2}$};
        \node (w6) at (10,-1.3) {\tiny$\bar{1}$};

        \draw (v1)--(v2) (v3)--(v4) (v5)--(v6);
        \draw (v2)--(v3) (v4)--(v5);
        \draw (v1) to [out=45, in=135] (v3) (v2) to [out=45, in=135] (v4);
        \draw (v3) to [out=45, in=135] (v5) (v4) to [out=45, in=135] (v6);

\end{tikzpicture}\noindent
    \caption{$\Gamma(\pi(I))_{e}\subset \Gamma(\pi(I))$}
\end{subfigure}\\
    \caption{}
    \label{fig:B=C}
\end{figure}

Next, consider $v=1~2~\bar{3}$ or $v=1~\bar{3}~2$ in ${}^JW$. Figure~\ref{fig:B!=C} displays the graphs $\Gamma(I)_v$ and $\Gamma(\pi(I))_v$. In either case, $\Gamma(I)_v$ and $\Gamma(\pi(I))_v$ are distinct, and, in fact, $\Gamma(I)_v$ can be viewed as a proper subgraph of $\Gamma(\pi(I))_v$, which by Lemma \ref{lem.triv}, implies
\[
    2=2\cdot 1\cdot 1=\langle \omega X_{\Gamma(I)_{12\bar{3}}},s_3 \rangle<\langle \omega X_{\Gamma(\pi(I))_{12\bar{3}}},s_3 \rangle=2\cdot 2\cdot 1=4
\]
and
\[
    2=2\cdot 1\cdot 1=\langle \omega X_{\Gamma(I)_{1\bar{3}2}},s_3 \rangle<\langle \omega X_{\Gamma(\pi(I))_{1\bar{3}2}},s_3 \rangle=2\cdot 2\cdot 1=4.
\]
This confirms the conclusion of Theorem~\ref{thm.nate} that the dot action  representations $H^*(\Hess(s, H(I)))$ and $H^*(s', H(\pi(I)))$ are not isomorphic in this case. 
\begin{figure}[ht]
    \centering
    \begin{subfigure}{0.4\textwidth}
    \begin{tikzpicture}[
        scale=0.55,
        vertex/.style={circle, draw, minimum size=0.45 cm},
        every edge/.style={draw, thick}
        ]
        \coordinate[circle, minimum size=0.45 cm] (p1) at (0,0);
        \coordinate[circle, minimum size=0.45 cm] (p2) at (2,0);
        \coordinate[circle, minimum size=0.45 cm] (p3) at (4,0);
        \coordinate[circle, minimum size=0.45 cm] (p0) at (6,0);
        \coordinate[circle, minimum size=0.45 cm] (p4) at (8,0);
        \coordinate[circle, minimum size=0.45 cm] (p5) at (10,0);
        \coordinate[circle, minimum size=0.45 cm] (p6) at (12,0);

        \path[draw=blue!20, line width=10,line cap=round, line join=round] 
        (p1) to (p2);
            
        %Vertices
        \node[vertex, fill=blue!20, label=below:{}] (v1) at (0,0) {\tiny$1$};
        \node[vertex, fill=blue!20, label=below:{}] (v2) at (2,0) {\tiny$2$};
        \node[vertex, label=below:{}] (v3) at (4,0) {\tiny$3$};
        \node[vertex, label=below:{}] (v0) at (6,0) {\tiny$0$};
        \node[vertex, fill=blue!20, label=below:{}] (v4) at (8,0) {\tiny $\bar{3}$};
        \node[vertex, label=below:{}] (v5) at (10,0) {\tiny $\bar{2}$};
        \node[vertex, label=below:{}] (v6) at (12,0) {\tiny$\bar{1}$};

        %\node (w) at (-1.3,-1.3) {$v=$};
        \node[circle, fill=red!20, minimum size=0.25 cm] (w1) at (0,-1.3) {\tiny$1$};
        \node[circle, fill=red!20 , minimum size=0.25 cm] (w2) at (2,-1.3) {\tiny$2$};
        \node (w3) at (4,-1.3) {\tiny$\bar{3}$};
        \node[circle, fill=red!20 , minimum size=0.25 cm] (w4) at (8,-1.3) {\tiny$3$};
        \node (w5) at (10,-1.3) {\tiny$\bar{2}$};
        \node (w6) at (12,-1.3) {\tiny$\bar{1}$};

        \draw (v1)--(v2) (v3)--(v0) (v0)--(v4) (v5)--(v6);
        \draw (v2)--(v3) (v4)--(v5);
        \draw (v1) to [out=45, in=135] (v3) (v2) to [out=45, in=135] (v0);
        \draw (v0) to [out=45, in=135] (v5) (v4) to [out=45, in=135] (v6);

\end{tikzpicture}\noindent
    \caption{$\Gamma(I)_{12\bar{3}}\subset \Gamma(I)$}  
\end{subfigure}
    \hfill
\begin{subfigure}{0.4 \textwidth}
    \begin{tikzpicture}[
        scale=0.55,
        vertex/.style={circle, draw, minimum size=0.45 cm},
        every edge/.style={draw, thick}
        ]
        \coordinate[circle, minimum size=0.45 cm] (p1) at (0,0);
        \coordinate[circle, minimum size=0.45 cm] (p2) at (2,0);
        \coordinate[circle, minimum size=0.45 cm] (p3) at (4,0);
        \coordinate[circle, minimum size=0.45 cm] (p4) at (6,0);
        \coordinate[circle, minimum size=0.45 cm] (p5) at (8,0);
        \coordinate[circle, minimum size=0.45 cm] (p6) at (10,0);

        \path[draw=blue!20, line width=10,line cap=round, line join=round] 
        (p1) to (p2) (p2) to [out=45, in=135] (p4) ;
            
        %Vertices
        \node[vertex, fill=blue!20, label=below:{}] (v1) at (0,0) {\tiny$1$};
        \node[vertex, fill=blue!20, label=below:{}] (v2) at (2,0) {\tiny$2$};
        \node[vertex, label=below:{}] (v3) at (4,0) {\tiny$3$};
        \node[vertex, fill=blue!20, label=below:{}] (v4) at (6,0) {\tiny $\bar{3}$};
        \node[vertex, label=below:{}] (v5) at (8,0) {\tiny $\bar{2}$};
        \node[vertex, label=below:{}] (v6) at (10,0) {\tiny$\bar{1}$};

        %\node (w) at (-1.3,-1.3) {$v=$};
        \node[circle, fill=red!20, minimum size=0.25 cm] (w1) at (0,-1.3) {\tiny$1$};
        \node[circle, fill=red!20 , minimum size=0.25 cm] (w2) at (2,-1.3) {\tiny$2$};
        \node (w3) at (4,-1.3) {\tiny$\bar{3}$};
        \node[circle, fill=red!20 , minimum size=0.25 cm] (w4) at (6,-1.3) {\tiny$3$};
        \node (w5) at (8,-1.3) {\tiny$\bar{2}$};
        \node (w6) at (10,-1.3) {\tiny$\bar{1}$};

        \draw (v1)--(v2) (v3)--(v4) (v5)--(v6);
        \draw (v2)--(v3) (v4)--(v5);
        \draw (v1) to [out=45, in=135] (v3) (v2) to [out=45, in=135] (v4);
        \draw (v3) to [out=45, in=135] (v5) (v4) to [out=45, in=135] (v6);

\end{tikzpicture}\noindent
    \caption{$\Gamma(\pi(I))_{12\bar{3}}\subset \Gamma(\pi(I))$}
\end{subfigure}\\
\begin{subfigure}{0.4\textwidth}
    \begin{tikzpicture}[
        scale=0.55,
        vertex/.style={circle, draw, minimum size=0.45 cm},
        every edge/.style={draw, thick}
        ]
        \coordinate[circle, minimum size=0.45 cm] (p1) at (0,0);
        \coordinate[circle, minimum size=0.45 cm] (p2) at (2,0);
        \coordinate[circle, minimum size=0.45 cm] (p3) at (4,0);
        \coordinate[circle, minimum size=0.45 cm] (p0) at (6,0);
        \coordinate[circle, minimum size=0.45 cm] (p4) at (8,0);
        \coordinate[circle, minimum size=0.45 cm] (p5) at (10,0);
        \coordinate[circle, minimum size=0.45 cm] (p6) at (12,0);

        \path[draw=blue!20, line width=10,line cap=round, line join=round] 
        (p1) to [out=45, in=135] (p3);
            
        %Vertices
        \node[vertex, fill=blue!20, label=below:{}] (v1) at (0,0) {\tiny$1$};
        \node[vertex, label=below:{}] (v2) at (2,0) {\tiny$2$};
        \node[vertex, fill=blue!20, label=below:{}] (v3) at (4,0) {\tiny$3$};
        \node[vertex, label=below:{}] (v0) at (6,0) {\tiny$0$};
        \node[vertex, label=below:{}] (v4) at (8,0) {\tiny $\bar{3}$};
        \node[vertex, fill=blue!20, label=below:{}] (v5) at (10,0) {\tiny $\bar{2}$};
        \node[vertex, label=below:{}] (v6) at (12,0) {\tiny$\bar{1}$};

        %\node (w) at (-1.3,-1.3) {$v=$};
        \node[circle, fill=red!20, minimum size=0.25 cm] (w1) at (0,-1.3) {\tiny$1$};
        \node (w2) at (2,-1.3) {\tiny$\bar{3}$};
        \node[circle, fill=red!20 , minimum size=0.25 cm] (w3) at (4,-1.3) {\tiny$2$};
        \node (w4) at (8,-1.3) {\tiny$\bar{2}$};
        \node[circle, fill=red!20 , minimum size=0.25 cm] (w5) at (10,-1.3) {\tiny$3$};
        \node (w6) at (12,-1.3) {\tiny$\bar{1}$};

        \draw (v1)--(v2) (v3)--(v0) (v0)--(v4) (v5)--(v6);
        \draw (v2)--(v3) (v4)--(v5);
        \draw (v1) to [out=45, in=135] (v3) (v2) to [out=45, in=135] (v0);
        \draw (v0) to [out=45, in=135] (v5) (v4) to [out=45, in=135] (v6);

\end{tikzpicture}\noindent
    \caption{$\Gamma(I)_{1\bar{3}2}\subset \Gamma(I)$}  
\end{subfigure}
    \hfill
\begin{subfigure}{0.4 \textwidth}
    \begin{tikzpicture}[
        scale=0.55,
        vertex/.style={circle, draw, minimum size=0.45 cm},
        every edge/.style={draw, thick}
        ]
        \coordinate[circle, minimum size=0.45 cm] (p1) at (0,0);
        \coordinate[circle, minimum size=0.45 cm] (p2) at (2,0);
        \coordinate[circle, minimum size=0.45 cm] (p3) at (4,0);
        \coordinate[circle, minimum size=0.45 cm] (p4) at (6,0);
        \coordinate[circle, minimum size=0.45 cm] (p5) at (8,0);
        \coordinate[circle, minimum size=0.45 cm] (p6) at (10,0);

        \path[draw=blue!20, line width=10,line cap=round, line join=round] 
        (p1) to [out=45, in=135] (p3) (p3) to [out=45, in=135] (p5) ;
            
        %Vertices
        \node[vertex, fill=blue!20, label=below:{}] (v1) at (0,0) {\tiny$1$};
        \node[vertex, label=below:{}] (v2) at (2,0) {\tiny$2$};
        \node[vertex, fill=blue!20, label=below:{}] (v3) at (4,0) {\tiny$3$};
        \node[vertex, label=below:{}] (v4) at (6,0) {\tiny $\bar{3}$};
        \node[vertex, fill=blue!20, label=below:{}] (v5) at (8,0) {\tiny $\bar{2}$};
        \node[vertex, label=below:{}] (v6) at (10,0) {\tiny$\bar{1}$};

        %\node (w) at (-1.3,-1.3) {$v=$};
        \node[circle, fill=red!20, minimum size=0.25 cm] (w1) at (0,-1.3) {\tiny$1$};
        \node (w2) at (2,-1.3) {\tiny$\bar{3}$};
        \node[circle, fill=red!20 , minimum size=0.25 cm] (w3) at (4,-1.3) {\tiny$2$};
        \node (w4) at (6,-1.3) {\tiny$\bar{2}$};
        \node[circle, fill=red!20 , minimum size=0.25 cm] (w5) at (8,-1.3) {\tiny$3$};
        \node (w6) at (10,-1.3) {\tiny$\bar{1}$};

        \draw (v1)--(v2) (v3)--(v4) (v5)--(v6);
        \draw (v2)--(v3) (v4)--(v5);
        \draw (v1) to [out=45, in=135] (v3) (v2) to [out=45, in=135] (v4);
        \draw (v3) to [out=45, in=135] (v5) (v4) to [out=45, in=135] (v6);

\end{tikzpicture}\noindent
    \caption{$\Gamma(\pi(I))_{1\bar{3}2}\subset \Gamma(\pi(I))$}
\end{subfigure}\\
    \caption{}
    \label{fig:B!=C}
\end{figure}

\end{example}

To conclude, we count the number of ideals with equal reflection sets, which is exactly the $(n+1)$-st Catalan number. As in the previous section, we let $J = \{e_1-e_2, e_2-e_3, \ldots, e_{n-1}-e_n\}$ and $\Phi_J^+ = \{e_i-e_j \mid 1\leq i <j \leq n\}$. 

\begin{lemma} \label{lem: Catalan}
The number of ideals $I\subseteq \Phi^+(B_n)$ such that $R(I) = R(\pi(I))$ is $\mathrm{Cat}_{n+1} := \frac{1}{n+1}{2n+2 \choose n+1}$.
\end{lemma}
\begin{proof} It is well-known that the number of naturally labeled unit interval graphs on $n$ vertices is $\mathrm{Cat}_n$. Given an ideal $I\subseteq \Phi^+(B_n)$, let $\tilde{\Gamma}(I)$ denote the induced subgraph of $\Gamma(I)$ on vertices $[n]\cup \{0\}$. For example, for the ideal $I$ and graph $\Gamma(I)$ as in Figure~\ref{fig:B3}, the graph $\tilde{\Gamma}(I)$ is pictured below. 
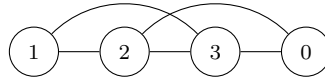
\begin{center}
        \begin{adjustbox}{max width=\textwidth}
            \begin{tikzpicture}[
                vertex/.style={circle, draw, minimum size=0.65 cm, inner sep=0pt},
                every edge/.style={draw, thick}
            ]
                %Vertices
                \node[vertex] (v1) at (0,0) {\tiny $1$};
                \node[vertex] (v2) at (1.2,0) {\tiny $2$};
                \node[vertex] (v3) at (2.4,0) {\tiny $3$};
                \node[vertex] (v7) at (3.6,0) {\tiny $0$};
                
                \draw (v1)--(v2)--(v3)--(v7);
                \draw (v1) to [out=45, in=135] (v3);
                \draw (v2) to [out=45, in=135] (v7);
            \end{tikzpicture}   
        \end{adjustbox}
\end{center}
It is straightforward to prove that $\tilde{\Gamma}(I)$ is a naturally labeled unit interval graph on $n+1$ vertices and that every such graph is of the form $\tilde{\Gamma}(I)$ for some ideal $I$ in $\Phi^+(B_n)$.

The edges of $\tilde\Gamma(I)$ are in bijection with the roots in the set $I\cap (\Phi_J^+ \sqcup \{e_1, \ldots, e_n\})$.

Let $I_1, I_2 \subseteq \Phi^+(B_n)$ be ideals such that $R(I_1) = R(\pi(I_1))$ and $R(I_2)=R(\pi(I_2))$. To complete the proof, it suffices to show that $I_1=I_2$ if and only if $\tilde{\Gamma}(I_1) = \tilde{\Gamma}(I_2)$. The forward direction is immediate, as $\tilde{\Gamma}(I_1) \neq \tilde{\Gamma}(I_2)$ implies $I_1\cap (\Phi_J^+ \sqcup\{e_1, \ldots, e_n\})\neq I_2\cap (\Phi_J^+ \sqcup\{e_1, \ldots, e_n\})$ and thus $I_1\neq I_2$. 

To prove the converse, suppose $I_1\neq I_2$. If $I_1\cap (\Phi_J^+ \sqcup\{e_1, \ldots, e_n\})\neq I_2\cap (\Phi_J^+ \sqcup\{e_1, \ldots, e_n\})$, then the graphs $\tilde\Gamma(I_1)$ and $\tilde{\Gamma}(I_2)$ have different edge sets so $\tilde{\Gamma}(I_1) \neq \tilde{\Gamma}(I_2)$. Thus we may assume that $I_1\cap (\Phi_J^+ \sqcup\{e_1, \ldots, e_n\})= I_2\cap (\Phi_J^+ \sqcup\{e_1, \ldots, e_n\})$. We show this assumption leads to a contradiction. 

Our assumption that $I_1$ and $I_2$ have the same intersection with $\Phi_J^+ \sqcup\{e_1, \ldots, e_n\}$ implies that if $\gamma\in I_1$ and $\gamma\notin I_2$ then $\gamma = e_i+e_j$ for some $1\leq i<j\leq n$. Since $e_i+e_j\geq e_i$ in the $B_n$ root system we have $e_i\in I_1$ as well. Since $I_1$ and $I_2$ have the same intersection with $\{e_1, \ldots, e_n\}$ we get $e_i\in I_2$ as well. Now $R(I_2) = R(\pi(I_2))$ implies $2e_i \in \pi(I_2)$. But then $e_i+e_j\in \pi(I_2)$ since $e_i+e_j\leq 2e_i$ in the type $C_n$ root system. Finally, we get $e_i+e_j\in I_2$, again because $R(I_2) = R(\pi(I_2))$, a contradiction. 
\end{proof}

The proof of Lemma \ref{lem: Catalan} also yields a natural bijection between the set of ideals $I\subseteq\Phi^+(B_n)$ such that $R(I)=R(\pi(I))$ and the set of ideals in $\Phi^+(A_n)$.
Consider the poset quotient map $p:\Phi^+(B_n)\rightarrow \Phi^+(A_n)$ defined by
\[
    p(\alpha)\coloneqq\begin{cases}
        e_i-e_j & \text{ if $\alpha=e_i-e_j$ for }1\le i<j\le n-1 \\
        e_i-e_{n+1} & \text{ if $\alpha\in [e_i,e_i+e_{i+1}]$ for }1\le i\le n-1 \\
        e_n-e_{n+1} & \text{ if $\alpha=e_n$}
    \end{cases}.
\]
See Figure \ref{fig: quotient} for the quotient map when $n=3$.
The bijection is given by pulling back any order ideal in $\Phi^+(A_n)$ through the projection $p$. 
Indeed, for any ideal $I\subseteq \Phi^+(A_n)$, its preimage $p^{-1}(I)$ is an order ideal in $\Phi^+(B_n)$ satisfying $R(p^{-1}(I))=R(\pi( p^{-1}(I)))$.
Moreover, any ideal in $\Phi^+(B_n)$ satisfying this condition arises uniquely as such a preimage.

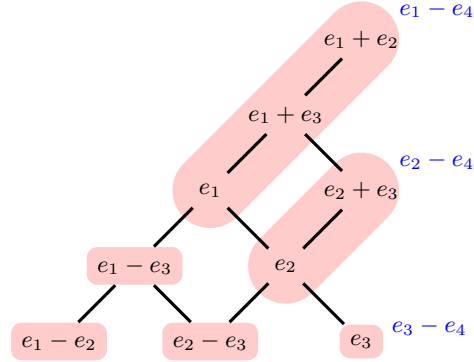
\begin{figure}[htbp]
    \begin{adjustbox}{max width=\textwidth}
        \begin{tikzpicture}[
            vertex/.style={font=\scriptsize}
        ]
            \def\df{9}
            \def\w{1 cm}
    
            \coordinate (100) at (0,0);
            \coordinate (010) at (2,0);
            \coordinate (001) at (4,0);
            \coordinate (110) at (1,1);
            \coordinate (011) at (3,1);
            \coordinate (111) at (2,2);
            \coordinate (021) at (4,2);
            \coordinate (121) at (3,3);
            \coordinate (221) at (4,4);
            \coordinate (e34) at (4.9,0.2);
            \coordinate (e24) at (5,2.4);
            \coordinate (e14) at (5,4.4);
                    
            \path[fill=blue!20, draw=red!20, line width=\w,line cap=round, line join=round] 
            (111) to (121) to (221);
            \path[fill=blue!20, draw=red!20, line width=\w,line cap=round, line join=round] 
            (011) to (021);
    
            \node[vertex, rounded corners, fill=red!20] (p100) at (100) {$e_1-e_2$};
            \node[vertex, rounded corners, fill=red!20] (p010) at (010) {$e_2-e_3$};
            \node[vertex, rounded corners, fill=red!20] (p001) at (001) {$e_3$};
            \node[vertex, rounded corners, fill=red!20] (p110) at (110) {$e_1-e_3$};
            \node[vertex] (p011) at (011) {$e_2$};
            \node[vertex] (p111) at (111) {$e_1$};
            \node[vertex] (p121) at (121) {$e_1+e_3$};
            \node[vertex] (p021) at (021) {$e_2+e_3$};
            \node[vertex] (p221) at (221) {$e_1+e_2$};
            \node[vertex] at (e34) {\textcolor{blue}{$e_3-e_4$}};
            \node[vertex] at (e24) {\textcolor{blue}{$e_2-e_4$}};
            \node[vertex] at (e14) {\textcolor{blue}{$e_1-e_4$}};
                
            \draw[very thick] (p100)--(p110)--(p111)--(p121)--(p221);
            \draw[very thick] (p110)--(p010)--(p011)--(p021)--(p121);
            \draw[very thick] (p001)--(p011)--(p111);
        \end{tikzpicture}
    \end{adjustbox}
    \caption{The equivalence classes in $\Phi^+(B_3)$ induced from $p$ and their images}
    \label{fig: quotient}
\end{figure}

%%%%%%%%%%%%%%%%%%

\bibliographystyle{alpha}
\bibliography{ref.bib}

\end{document}